\documentclass{article}
\usepackage{maa-monthly-noheader}
\final

\newtheorem*{thm*}{Theorem}
\newtheorem{thm}{Theorem}[section]
\newtheorem{lem}[thm]{Lemma}
\newtheorem{prop}[thm]{Proposition}
\newtheorem{cor}[thm]{Corollary}

\theoremstyle{definition}
\newtheorem{definition}[thm]{Definition}

\newtheorem{rem}[thm]{Remark}

\newcommand{\Z}{\mathbb{Z}}
\newcommand{\Q}{\mathbb{Q}}
\newcommand{\R}{\mathbb{R}}
\newcommand{\C}{\mathbb{C}}
\newcommand{\san}{\emph{sangaku}}
\newcommand{\Gal}{\operatorname{Gal}}
\newcommand{\Res}{\operatorname{Res}}
\newcommand\aang{\theta}
\newcommand\aalft{\beta}
\newcommand\aalftz{\beta_0}
\newcommand\cca{C_1}
\newcommand\ccb{C_r}
\newcommand\con{Con}
\newcommand\fa{g}
\newcommand\fb{f_4}
\newcommand\fc{f_3}
\newcommand\fd{f_5}
\newcommand\fe{f_1}
\newcommand\ff{f_2}
\newcommand\fg{f_6}
\newcommand\fh{f_7}
\newcommand\lli{L}
\newcommand\ooa{O_1}
\newcommand\oob{O_r}
\newcommand\rt{k}

\newcommand\sside{s}

\newcommand\tta{T_1}
\newcommand\ttb{T_r}
\newcommand\ttl{T_L}
\newcommand\vva{V_1}
\newcommand\vvb{V_r}
\newcommand\vvdn{V_\textsubscript{dn}}
\newcommand\vvup{V_\textsubscript{up}}
\newcommand\xt{x}
\newcommand\yp{y_{_P}}
\begin{document}

\title{Morikawa's Unsolved Problem}
\markright{Morikawa's Unsolved Problem}
\author{Jan E. Holly and David Krumm}
\maketitle

\begin{abstract}
By combining theoretical and computational techniques from geometry, calculus, group theory, and Galois theory, we prove the nonexistence of a closed-form algebraic solution to a Japanese geometry problem first stated in the early nineteenth century. This resolves an outstanding problem from the \san\, tablets which were at one time displayed in temples and shrines throughout Japan.
\end{abstract}

\section{Introduction.}\label{intro_section} 
During the Edo Period of Japanese history (1603--1867) there developed a curious practice of hanging wooden tablets with mathematical content from the eaves of Buddhist temples and Shinto shrines. Many of these tablets, known as \san, have been lost to history, but close to 900 of them have been preserved \cite{rothman}. The problems inscribed on the surviving \san\, are mostly of a geometric nature, and they range in difficulty from trivial to unsolved. Solutions to many of these problems can be found in the books by Fukagawa and Pedoe \cite{fukagawa-pedoe} and Fukagawa and Rothman \cite{fukagawa-rothman}; the latter reference also discusses various historical aspects surrounding the mathematics of the Edo Period.

Unsolved \san\ problems seem to be rare; in fact, we are aware of only two such problems listed in the literature, both in Fukagawa and Rothman \cite[Chapter~7]{fukagawa-rothman}. One of the problems mentioned in that text was originally proposed in 1821 and has recently been solved \cite{kinoshita}. The present article concerns the other unsolved problem, which was proposed by Jihei Morikawa during the same time period.

\begin{figure}[b]
  \centering
\includegraphics[width=0.85\linewidth]{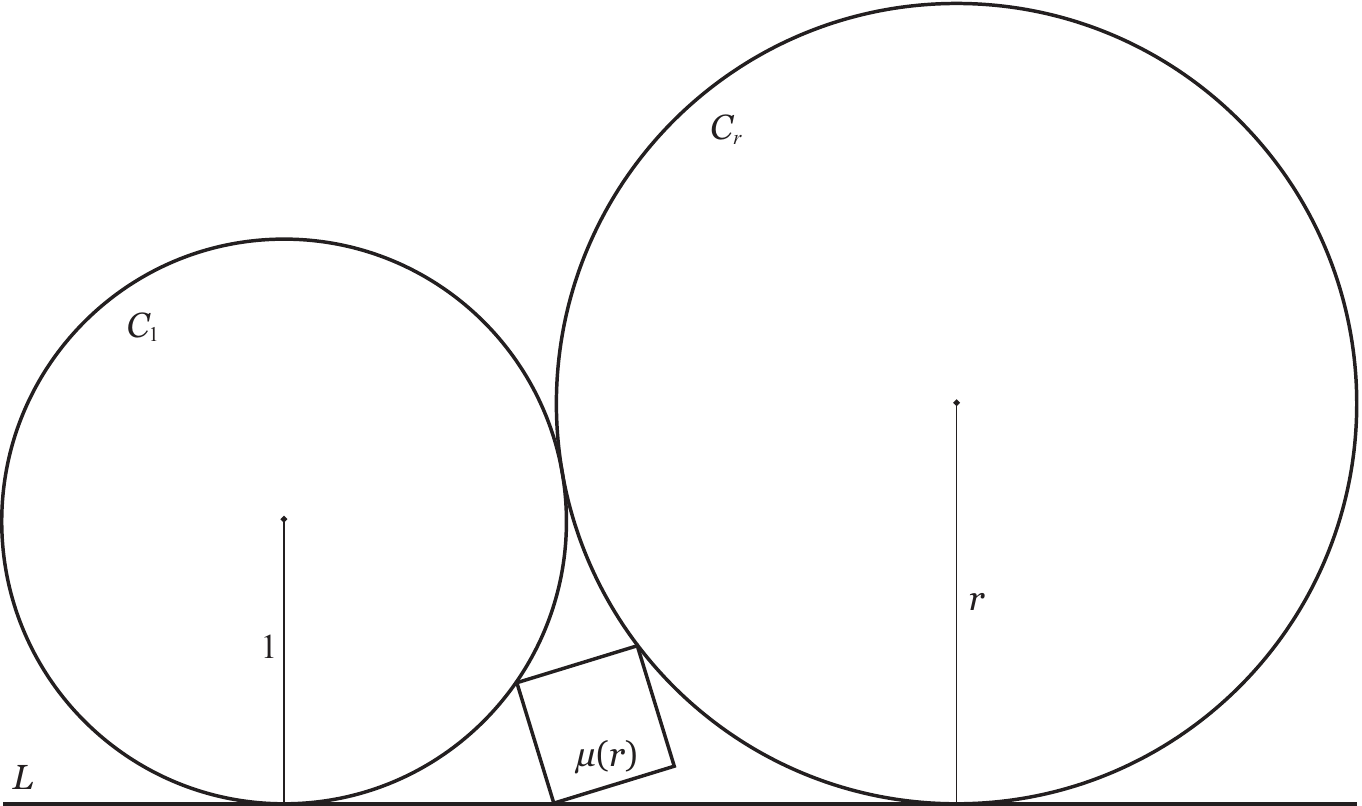}
\caption{The circles $C_1$ and $C_r$, the line $L$, and the minimal inscribed square with side length $\mu(r)$. The original statement of the problem involves circles of radii $a$ and $b$ with $b\ge a$, but a scaling of the plane reduces this general case to the case $a=1$.}
\label{morikawa_figure}
\end{figure}

The main objects involved in Morikawa's problem are illustrated in Figure~\ref{morikawa_figure}. Given a line $L$ and circles $C_1$ and $C_r$ of radii 1 and $r\ge 1$, respectively, such that $C_1$ and $C_r$ are tangent to each other and to $L$, the problem asks us to express, in terms of $r$, the minimum side length $\mu(r)$ of a square that can be inscribed in the region between $C_1,C_r$, and $L$.  Here, ``inscribed'' means touching all three of $C_1,C_r$, and $L$.

A surviving travel diary of mathematician Kanzan Yamaguchi, a contemporary of Morikawa, includes an entry  with some additional information about Morikawa's problem and the \san\ containing it. Based on that entry, Fukagawa and Rothman \cite[p.~265]{fukagawa-rothman} report the following.

\begin{quotation}
\noindent The tablet contained no solution, but Morikawa had written, ``I will be very happy if someone can solve this problem.'' And so, says Yamaguchi, ``I went to Morikawa's home with my friend Takeda and asked him what the answer is. He said that he could not solve the problem yet.'' Neither does Yamaguchi's diary contain a solution and, like Morikawa, we would be very happy if someone solves this problem. \cite[p.~265]{fukagawa-rothman}
\end{quotation}

It seems surprising that Morikawa's problem would have frustrated all attempts at a solution, considering that the mathematicians of the Edo Period had a strong understanding of geometry, were adept in the use of algebra, and even had some knowledge of basic calculus. One begins to wonder whether the problem can in fact be solved.

The purpose of this article is to address the question of the existence of a closed-form expression for $\mu(r)$. From our analysis in Section \ref{polynomial_section} it follows that $\mu(r)$ is a root of a polynomial whose coefficients are polynomials in $r$; in light of this fact, a natural question is whether $\mu(r)$ is expressible by radicals in terms of~$r$. Precise terminology is defined below, but the question can be stated intuitively as follows: Is there a radical expression, such as
\[
\frac{\sqrt[11]{3r-\pi}\cdot\sqrt[3]{r^5-r+2i}+\sqrt[4]{1+i+\sqrt[5]{r^6-7}}}{3-\sqrt{2r^3+9r-5}},
\]
that for every real number $r\ge 1$ can be evaluated to yield $\mu(r)$?

We provide here a negative answer to this question, thus showing that a closed-form algebraic solution does not exist in the classical sense.
In order to state our results we introduce the following terminology. Recall that if $J\subseteq\C$ is a nonempty set and $f:J\to\C$ is a function, we say that $f$ is an \emph{algebraic function} if there exists a nonzero polynomial $q\in\C[k,x]$ such that $q(c, f(c))=0$ for every $c\in J$. If, moreover, this condition is satisfied by a polynomial $q$ whose Galois group over the field $\C(k)$ is solvable --- or equivalently, whose splitting field is contained in a radical extension of~$\C(k)$ --- then we say that $f$ is a \emph{radical function}.

We can now state our main result.

\begin{thm*}[see Theorem \ref{main_thm}]
The function $\mu:[1,\infty)\to\R$ is not radical. In fact, there is no infinite subset $J\subseteq [1,\infty)$ such that $\mu:J\to\R$ is radical.
\end{thm*}

The proof of this theorem makes critical use of computational tools in Galois theory that have only recently become available due to work of N. Sutherland \cite{sutherland}. In particular, our argument relies on Sutherland's implementation in the system \textsc{Magma} \cite{magma} of an algorithm for computing geometric Galois groups. Besides this algorithm, the proof uses elementary geometry as well as calculus and Galois theory.

This article is organized as follows. In Section \ref{parametrization_section} we provide notation and basic results about inscribed squares.
In Section \ref{orientation_section} we show that any minimal inscribed square must be positioned as in Figure~\ref{morikawa_figure}, with a corner on each of $C_1, C_r$, and $L$, and with no side of the square tangent to these objects. In Section \ref{polynomial_section} we derive an explicit formula for a function whose minimum value is~$\mu(r)$; as a byproduct we obtain a numerical method for approximating $\mu(r)$ given the radius $r$. In addition, we show that $\mu(r)$ can be expressed in terms of a root of a certain polynomial of degree 10. Finally, in Section \ref{algebraic_section} we use Galois theory to study this polynomial and thus prove the main theorem.

\section{Configurations for inscribed squares.}\label{parametrization_section}
Figure~\ref{notation} shows the notation that will be used throughout the article, regardless of which inscribed square is under discussion.  Denoted are the circles ($\cca$,~$\ccb$), centers of the circles ($\ooa$,~$\oob$), line ($\lli$), vertices of the square ($\vva$, $\vvb$, $\vvup$, $\vvdn$), and angle between $\lli$ and the lower right side of the square ($\aang\in [0,\pi/2)$). If $\aang=0$, then $\vvdn$ is the lower left vertex. To facilitate phrasing, the line is considered horizontal as shown, with $\cca$ on the left.

Most of the discussion in Sections \ref{parametrization_section} and \ref{orientation_section} assumes an arbitrary but fixed value of~$r\geq 1$.  Basic geometric facts \cite{coxeter} are used throughout.

\begin{figure}[ht]
  \centering
  \includegraphics[width=0.85\linewidth]{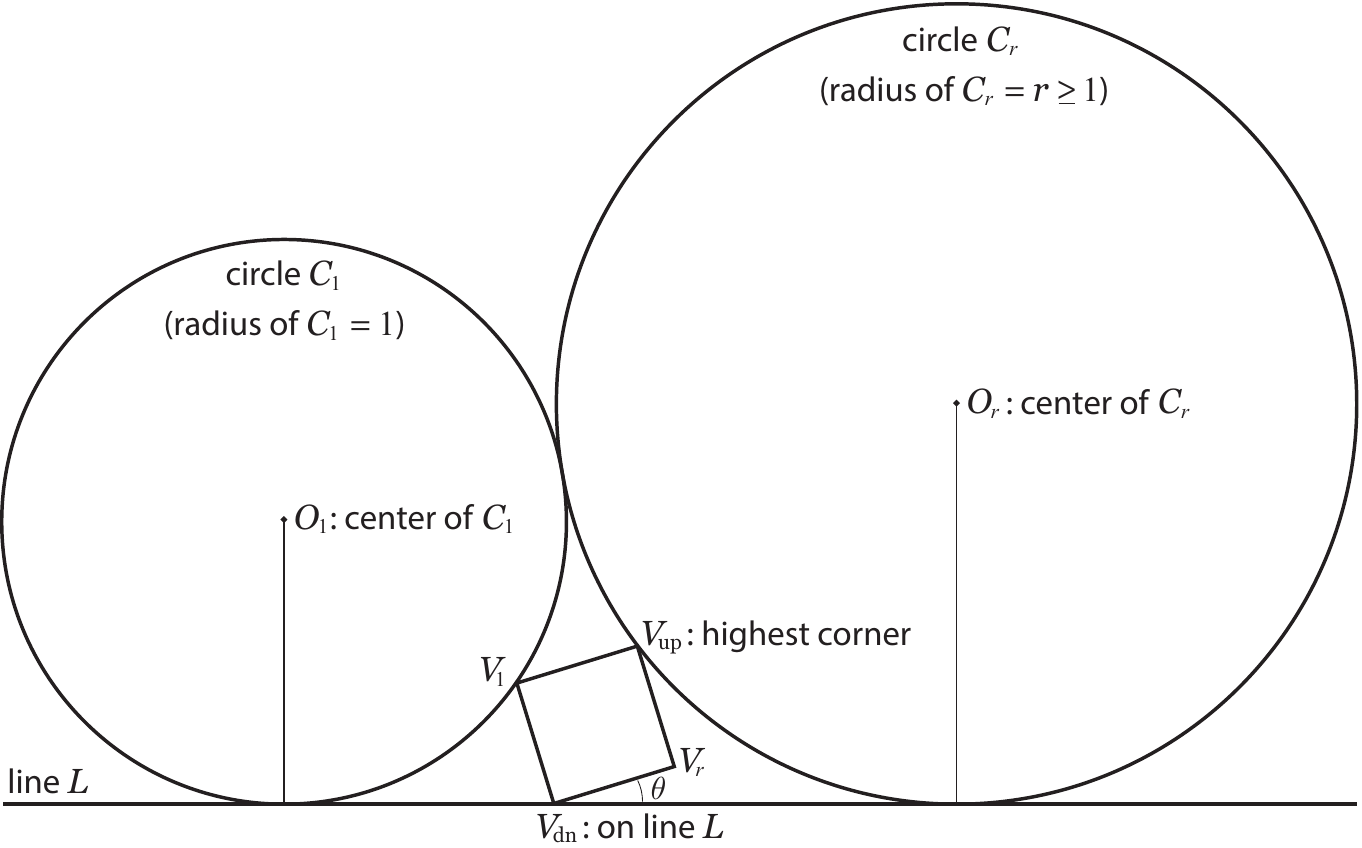}
  \caption{Notation for the line, two circles, and square.}
  \label{notation}
\end{figure}

\begin{lem}
For inscribed squares (for fixed $r\geq 1$), let $\aang\in [0,\pi/2)$ be as in Figure~\ref{notation}.
\begin{itemize}
\item[]
\begin{itemize}
\item[(i)]  For every $\aang\in [0,\pi/2)$, there is a unique inscribed square.
\item[(ii)]  There exists a minimum side length over the set of inscribed squares.  (Thus, Morikawa's problem is well-defined.)
\end{itemize}
\end{itemize}
Note:  Throughout the rest of the article, the side length of the inscribed square at angle $\aang\in [0,\pi/2)$ will be denoted~$s(\aang)$.
\end{lem}
\begin{proof}
To prove (i), fix $\aang\in [0,\pi/2)$.  We find the inscribed square at angle $\aang$ as follows, stated somewhat informally to avoid excessive technicalities.

Consider all squares, inscribed or not, that make angle $\aang$ with $\lli$ as in Figure~\ref{notation}.  Let $s$ be a side length under consideration for being that of an inscribed square.  Imagine sliding such a square --- with angle $\aang$ and side length~$s$ --- along $\lli$ until the square is to the right of $\cca$ but is just touching~$\cca$.  If $s$ is too small, then the square will not reach $\ccb$.  If $s$ is too large, then the square will overlap~$\ccb$.  By a continuity and monotonicity argument, there is a unique $s$ such that the $\cca$-touching square will exactly reach~$\ccb$.

The existence and uniqueness of an inscribed square follows from that of $s$ above, along with the fact that the inscribed square clearly cannot be moved left or right and still be inscribed.

Subsequently, (ii) follows from the facts that $s(\aang)$ is continuous on $[0,\pi/2)$ and $\lim_{\aang\to\pi/2^-} s(\aang) = s(0)$.
\end{proof}

Generally in this article, ``the square'' will mean the inscribed square as given by the context, unless stated otherwise.  Figure~\ref{configurations} shows the types of intersections between the square and circles as $\aang$ increases from $0$ to~$\pi/2$, i.e.,\ as the square rotates (and changes size as necessary).  Each combination of such intersections, as shown in Figure~\ref{configurations}, will henceforth be referred to as a \emph{configuration}.
The configurations as illustrated in Figure~\ref{configurations} will be denoted \con1, \con2, \con3, etc.

\begin{figure}[ht]
  \centering
  \includegraphics[width=0.9\linewidth]{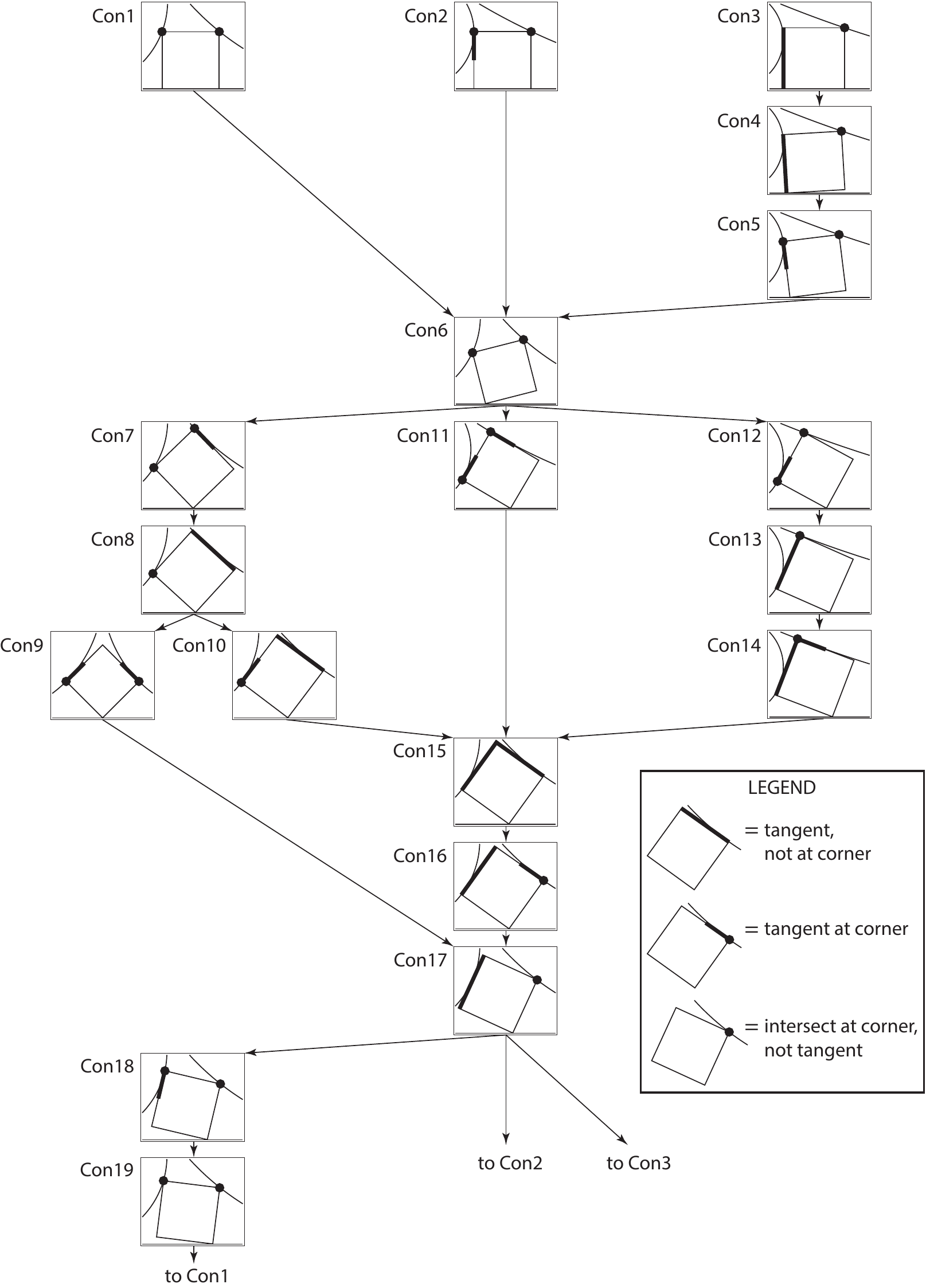}
  \caption{Possible configurations, in terms of the square's types of intersections with the line and circles. In Section \ref{orientation_section} we prove that a minimal square can occur only in \con6, or in \con19 with $r=1$, by eliminating all of the others: \con1,2,3 (Lemma~\ref{lem123}), \con4,5 (Lemma~\ref{lem456}), \con7,10,11 (Lemma~\ref{lem71011}), \con12,13,14 (Lemma~\ref{lem121314}), \con15 (Lemma~\ref{lem15}), \con16,18 (Lemma~\ref{lem1618}), \con19 (Lemma~\ref{lem19}), \con8,9,17 (Lemma~\ref{lem8917}).
  }
  \label{configurations}
\end{figure}

\begin{lem}
Figure~\ref{configurations} shows all possible steps through the configurations as $\aang$ increases from $0$ to~$\pi/2$.
\end{lem}
\begin{proof}
The steps through the configurations obviously depend upon~$r$.  For example, from \con6, the value of $r$ determines which circle first becomes tangent to a side of the square as $\aang$ increases.  Small $r$ leads to \con7, large $r$ leads to \con12, and a certain intermediate value of $r$ leads to \con11.  Also note that only $r=1$ gives \con9.

The fact that these are the steps through the configurations is generally clear, with two exceptions:  from \con8 to \con9 and \con10, and from \con15 to \con16.

For \con8 to \con9 and \con10, the question is whether the upper right side of the square could instead rotate past tangency with $\ccb$ before the upper left side of the square becomes tangent to~$\cca$.  This can happen only if the lower left side of the square has steeper slope than the line through $\vvdn$ and~$\ooa$ --- i.e.,\ the lower left side ``points above'' $\ooa$ --- and the lower right side points above~$\oob$.  However, this is impossible because the circle of radius $(r+1)/2$ through $\ooa$ and~$\oob$, as shown in Figure~\ref{tangentcircle}, is tangent to~$\lli$.  Every angle inscribed in a semicircle is a right angle, so since $\lli$ is below the new circle except at the point of tangency, $Q$, the lower sides of the square cannot both point above their respective circles' centers.  At best, the lower sides can point exactly at the centers, but only if $r=1$ and $\vvdn=Q$.

\begin{figure}[ht]
  \centering
  \includegraphics[width=\linewidth]{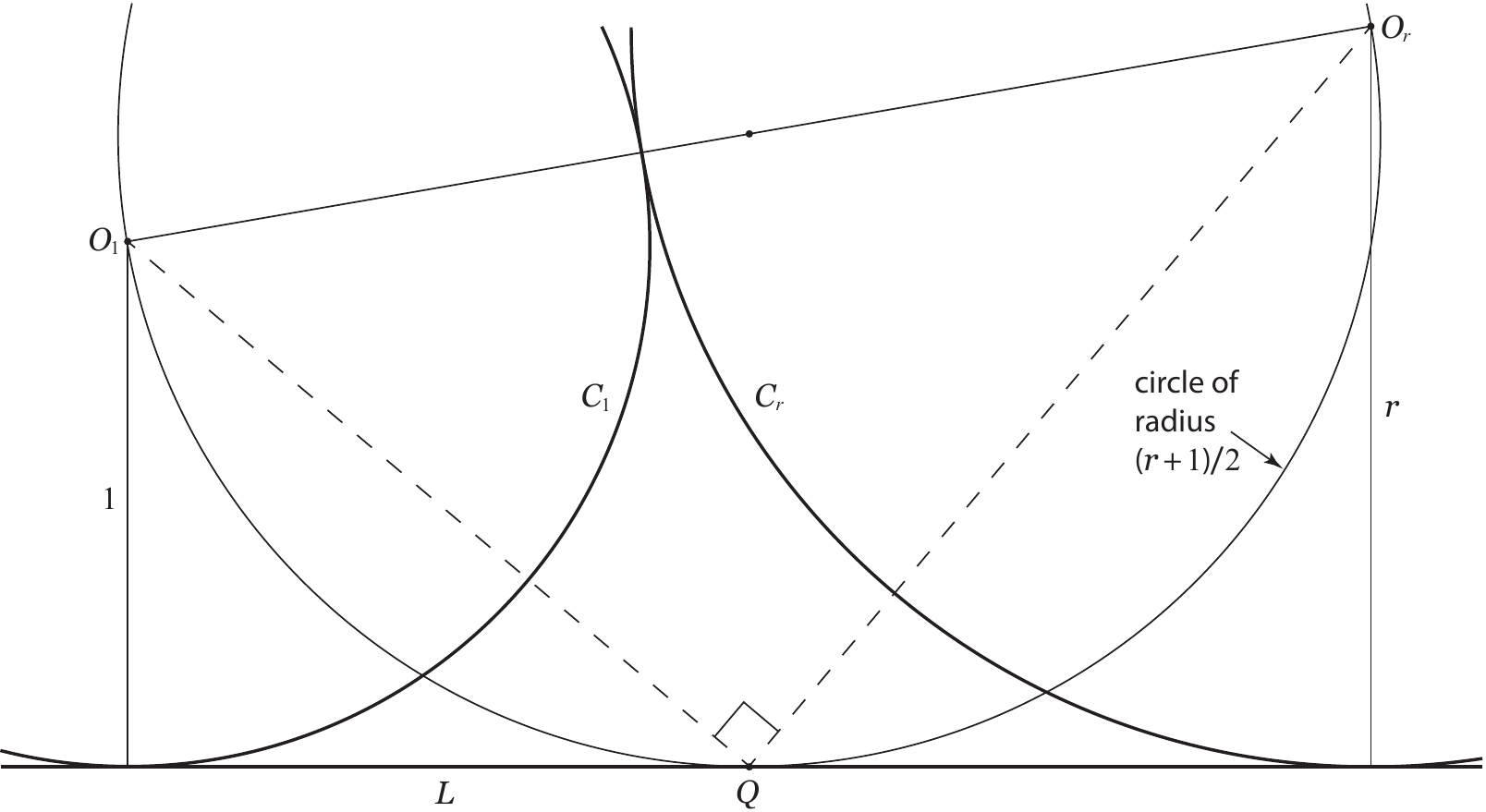}
  \caption{Circle with diameter $\ooa$-to-$\oob$ is tangent to~$L$.}
  \label{tangentcircle}
\end{figure}

For \con15 to \con16, the question is whether $\vvup$ can instead touch $\cca$ before the lower right side of the square touches~$\ccb$.  This can happen only if the square with upper right side tangent at $\vvb$ to $\ccb$ has $\vvup$ touching~$\cca$.  However, this is impossible.  Any square with upper right side tangent at $\vvb$ to $\ccb$ must have $\vvdn$ at or to the left of $Q$ in Figure~\ref{tangentcircle} in order for the square to reach~$\cca$, but then the square is angled such that $\vvup$ cannot touch~$\cca$.
\end{proof}

\section{Configurations for minimal squares.}\label{orientation_section}

In this section we prove that a minimal square --- i.e.,\ an inscribed square with side length that is minimal over \emph{all} orientations --- exists only in \con6, and has $\vva$ lower than~$\ooa$.  An exception occurs if $r=1$, where \con19 is the reflection of \con6 and thus also has a minimal square.  The proof consists of a sequence of lemmas showing that  a minimal square cannot be in any other configuration.  The final result is given by Proposition~\ref{prop6only}.

Additional notation is used throughout this section, for the lines perpendicular to each of $\cca$, $\ccb$, and $\lli$ at the points of contact with the square under consideration.  As illustrated in Figure~\ref{tnotation}, these (dashed) lines are denoted $\tta$, $\ttb$, and~$\ttl$, respectively.  As before, unless stated otherwise we assume an arbitrary but fixed value of $r\geq 1$.

\begin{lem}\label{triangle}
For a given inscribed square, if $\tta$ intersects $\ttl$ above~$\ttb$ (respectively, below~$\ttb$), then $s$ is a strictly decreasing (respectively, increasing) function of $\aang$ at that square's angle.
\end{lem}
\begin{proof}
If $\tta$ intersects $\ttl$ above $\ttb$, then the lines form a triangle to the right of $\ttl$, such as in Figure~\ref{tnotation}.

\begin{figure}[ht]
  \centering
  \includegraphics[width=0.85\linewidth]{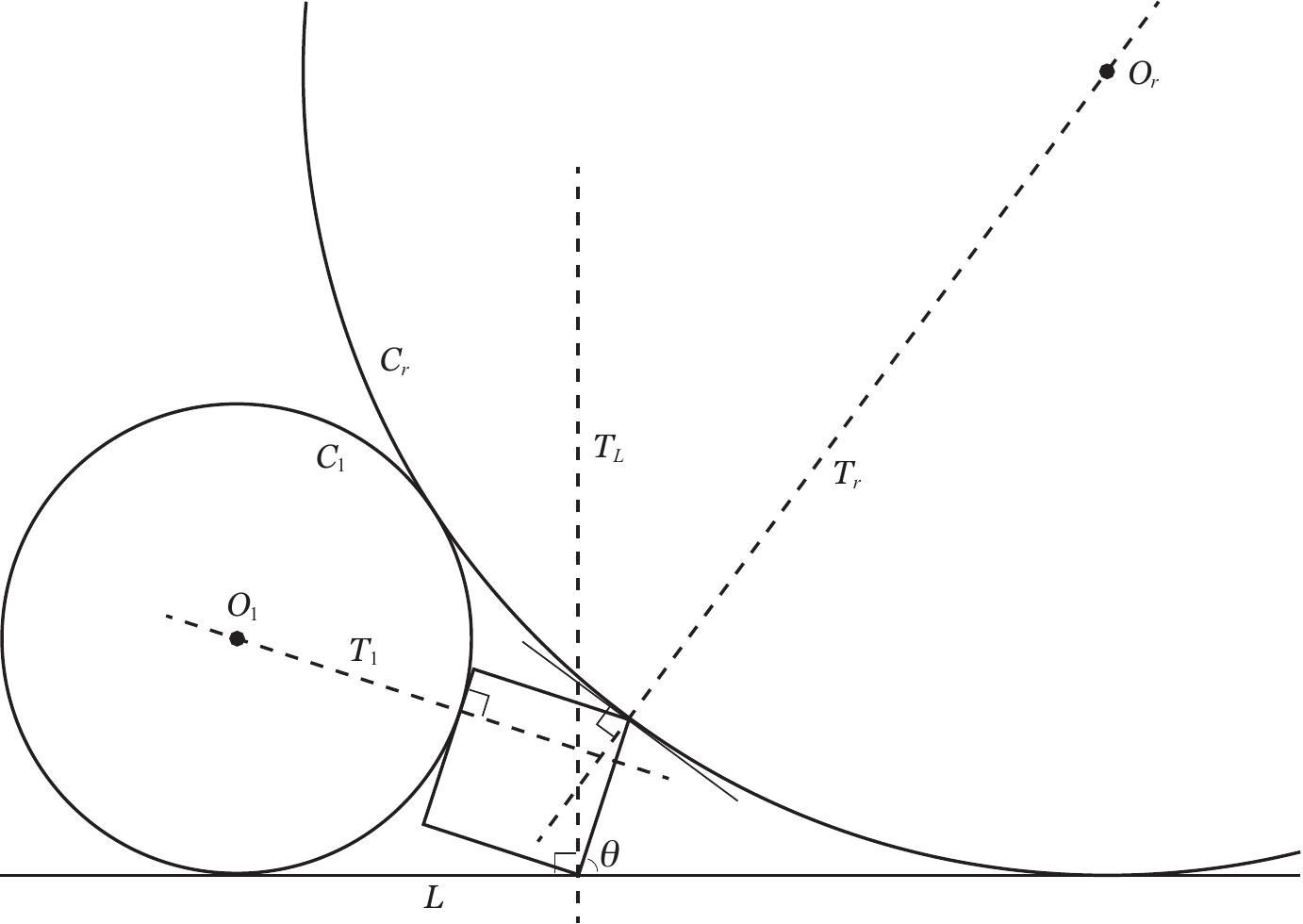}
  \caption{Notation: Lines $\tta$, $\ttb$, and $\ttl$ through points of intersection.}
  \label{tnotation}
\end{figure}

Consider fixing the size of the square, and rotating it counterclockwise about a point inside the triangle.  As the square begins to rotate, it starts to overlap with each of $\cca$, $\ccb$, and~$\lli$.  Formally, the rate of change of the following are positive:  (1)~the radius of $\cca$ minus the distance from $\ooa$ to the square, (2)~the radius of $\ccb$ minus the distance from $\oob$ to the square, and (3)~the distance below $\lli$ to the lowest point on the square.  This means that the difference between the side lengths of an inscribed square and a fixed-size square is a strictly decreasing function of~$\aang$; thus $s$ is a strictly decreasing function of $\aang$.  Similarly, if $\tta$ intersects $\ttl$ below $\ttb$, then $s$ is a strictly increasing function of $\aang$.
\end{proof}

\begin{lem}\label{lem123}
A minimal square cannot be in \con1, \con2, or \con3.
\end{lem}
\begin{proof}
Each of these configurations corresponds to $\aang=0$.  For small enough $\epsilon>0$, it is easy to see that for all $\aang\in (0,\epsilon)$, $\tta$ intersects $\ttl$ above~$\ttb$.  Therefore by Lemma~\ref{triangle}, $\sside(\aang)$ is strictly decreasing for $\aang\in (0,\epsilon)$ and thus by continuity $\sside(0)$ is not minimal.
\end{proof}

\begin{lem}\label{lem456}
A minimal square cannot be in \con4 or \con5, and a minimal square cannot be in \con6 with $\vva$ as high or higher than~$\ooa$.
\end{lem}
\begin{proof}
Figure~\ref{for456} illustrates \con4, and the same reasoning applies to the other two configurations.  First, note that $\vva$ is higher than $\ooa$ in \con4 and \con5 because some part of the left side of the square is tangent to~$\cca$, and that left side has negative slope.  Consider the line $L'$ tangent to $\ccb$ at $\vvup$, and the line $L''$ bisecting the angle between $L$ and~$L'$.  Line $L'$ must intersect $\cca$ (in order to ``escape'' the region between $\cca$, $\ccb$, and~$\lli$), so $L''$ is below~$\ooa$.

\begin{figure}[ht]
  \centering
  \includegraphics[width=0.9\linewidth]{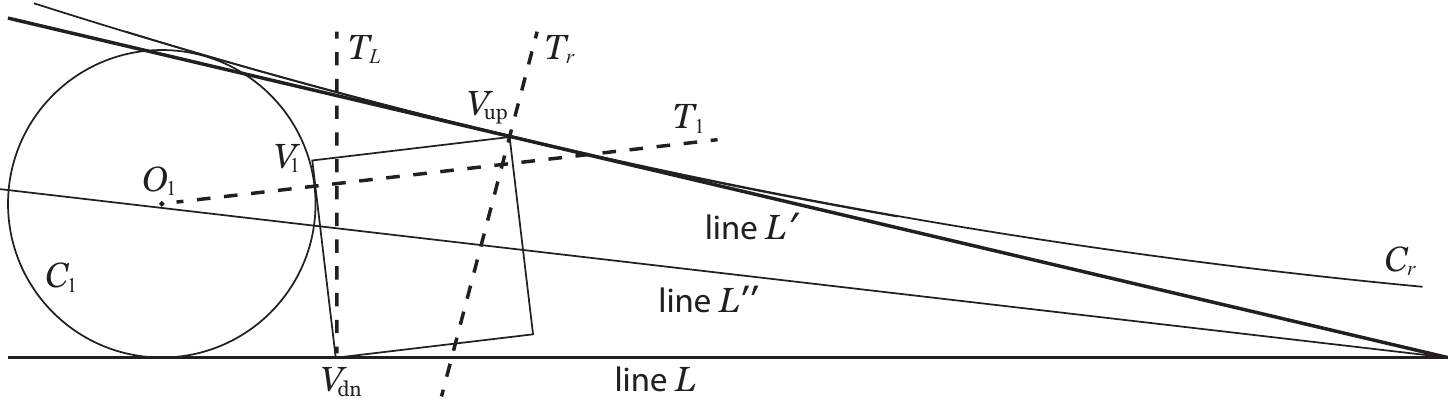}
  \caption{Illustration for Lemma~\ref{lem456}.  The line $L''$ bisecting the angle between $L$ and~$L'$ is below~$\ooa$, so $\vva$ is higher than~$L''$, leading to an application of Lemma~\ref{triangle}.}
  \label{for456}
\end{figure}

Because $\vva$ is higher than $\ooa$ and therefore higher than $L''$, the square is tilted in such a way that $\vvup$ is closer than $\vvdn$ to the intersection of $L'$ and~$L$.  Therefore, $\ttb$ intersects $L''$ to the right of where $\ttl$ intersects~$L''$; thus the intersection of $\ttl$ and~$\ttb$ is below the intersection of $\ttl$ and~$L''$.  Meanwhile, $\tta$~has nonnegative slope and therefore intersects $\ttl$ above the intersection of $\ttl$ and $L''$, thus above the intersection of $\ttl$ and~$\ttb$. Hence by Lemma~\ref{triangle}, $\sside(\aang)$~is a strictly decreasing function and is therefore not minimal for these values of~$\aang$.
\end{proof}

\begin{lem}\label{lem71011}
A minimal square cannot be in \con7, \con10, or \con11.
\end{lem}
\begin{proof}
In each of these configurations, $\tta$ intersects $\ttl$ below the intersection between $\ttl$ and~$\ttb$.  Therefore by Lemma~\ref{triangle}, $\sside(\aang)$ is strictly increasing and therefore not minimal at the corresponding~$\aang$.
\end{proof}

\begin{figure}[h]
  \centering
  \includegraphics[width=0.4\linewidth]{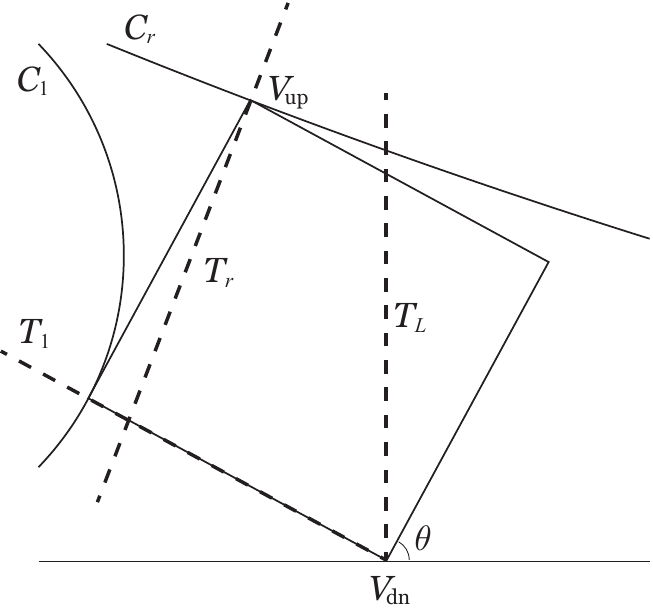}
  \caption{Illustration for Lemma~\ref{lem121314}.  In \con12, \con13, and \con14, $\aang>\pi/4$, so $\ttl$ is to the right of~$\vvup$, leading to an application of Lemma~\ref{triangle}.}
  \label{for121314}
\end{figure}

The proof that a minimal square cannot be in \con8, \con9, or \con17 is substantially more complicated than the other proofs, and is saved for the end of the section.

\begin{lem}\label{lem121314}
A minimal square cannot be in \con12, \con13, or \con14.
\end{lem}
\begin{proof}
In each of these configurations, some part of the upper left side of the square is tangent to~$\cca$.  Therefore, in order for $\vvup$ to reach~$\ccb$, $\aang$ cannot be particularly small, and certainly $\aang>\pi/4$:  As seen in \con9 (Figure~\ref{configurations}) which has $\aang=\pi/4$ and $r=1$, a square with $\aang=\pi/4$ and upper left side tangent to $\cca$ can only just barely reach --- and not even with~$\vvup$ --- the smallest possible $\ccb$, that with $r=1$.  Then decreasing~$\aang$ with the square's upper left side still tangent to~$\cca$ would shrink and move the square away from~$\ccb$, so $\aang\le\pi/4$ is not possible here when $r=1$.  It is also not possible when $r>1$ because any larger $\ccb$ would further prevent the square from reaching it.

Therefore, $\ttl$ is to the right of~$\vvup$, as illustrated in Figure~\ref{for121314} for the case of \con12.  Thus by Lemma~\ref{triangle}, $\sside(\aang)$~is a strictly increasing function of~$\aang$ and is therefore not minimal for these values of~$\aang$.
\end{proof}

\begin{lem}\label{lem15}
A minimal square cannot be in \con15.
\end{lem}
\begin{proof}
We prove this by showing that as the square rotates counterclockwise through the range of \con15, the intersection between $\tta$ and $\ttl$ moves upward strictly monotonically, and the intersection between $\ttb$ and $\ttl$ moves downward strictly monotonically.  Because it is clear in \con15 that $\ttl$'s intersection with $\tta$ is below that with $\ttb$ for the smaller values of~$\aang$, and above that with $\ttb$ for the larger values of~$\aang$, Lemma~\ref{triangle} implies that $\sside(\aang)$ strictly increases then strictly decreases as a function of~$\aang$, so the result follows.  (This seems to imply that there is a maximal square in \con15.  However, this article does not address maximal squares.)

We give the proof that the intersection between $\tta$ and $\ttl$ moves upward strictly monotonically;  an analogous proof shows that the intersection between $\ttb$ and $\ttl$ moves downward strictly monotonically.

\begin{figure}[ht]
  \centering
  \includegraphics[width=0.65\linewidth]{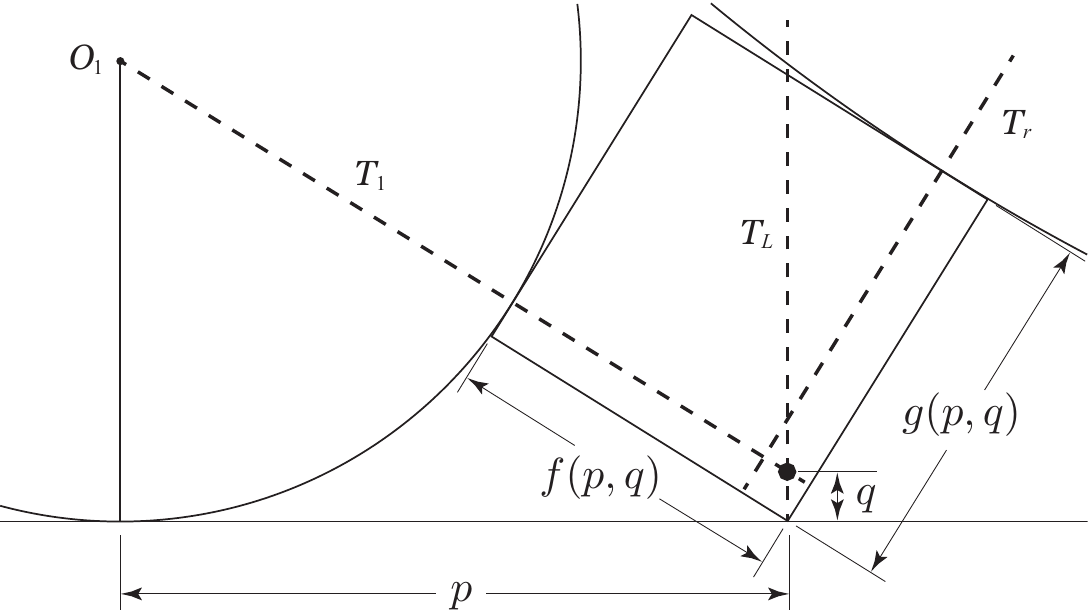}
  \caption{For Lemma~\ref{lem15}, the rectangle determined by the point~$(p,q)$.}
  \label{for15a}
\end{figure}

Consider not just inscribed squares, but inscribed rectangles in general, determined by the pair $(p,q)$ as shown in Figure~\ref{for15a}, where $p$ is the horizontal distance between $\ooa$ and the bottom vertex, and $q$ is the height above $L$ of the intersection between $\tta$ and~$\ttl$.  Let
\begin{align*}
f(p,q) &= \text{length of lower left side of rectangle,} \\
g(p,q) &= \text{length of lower right side of rectangle.}
\end{align*}
For each $p$ there is a unique $q$, say $q=h(p)$, such that the rectangle is a square, i.e.,\ $f(p,q)=g(p,q)$. Our goal is to show that $h(p)$ is a strictly increasing function of~$p$.  We will do this by showing that
\begin{itemize}
\item[]
\begin{itemize}
\item[(1)]  $f(p,q)$ is a strictly increasing function of~$p$,
\item[(2)]  $g(p,q)$ is a strictly decreasing function of~$p$,
\item[(3)]  $f(p,q)$ is a strictly decreasing function of~$q$,
\item[(4)]  $g(p,q)$ is a strictly increasing function of~$q$.
\end{itemize}
\end{itemize}
This will complete the proof because if point $(p,q)$ gives a square, then (1) and (2) imply that increasing $p$ will cause the lower left side to become larger than the lower right side, so by (3) and (4), $q$ must be increased in order to restore equality of the side lengths.

\begin{figure}[ht]
  \centering
  \includegraphics[width=0.65\linewidth]{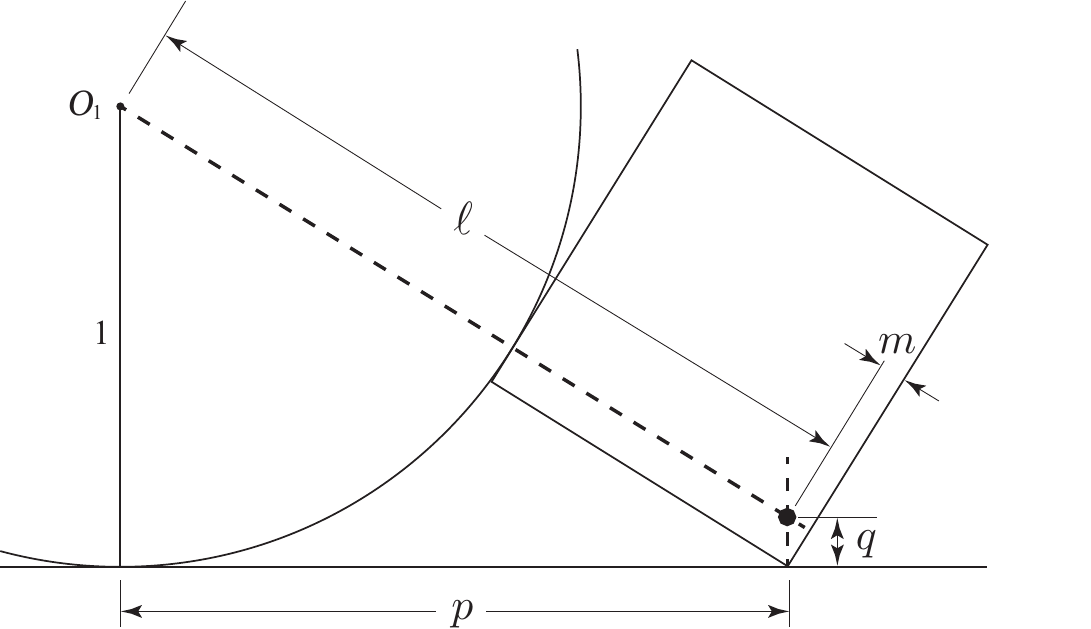}
  \caption{For Lemma~\ref{lem15}, notation for the proof that $f(p,q)$ is a strictly increasing function of~$p$.}
  \label{for15b}
\end{figure}

For (1), fix $q$ and note that $f(p,q)$ is $\ell+m-1$ in Figure~\ref{for15b}.  By similar triangles, $m/q=(1-q)/\ell$.  Define a function
\[
j(\ell) = \ell + m - 1 = \ell + \frac{q(1-q)}{\ell} - 1,
\]
which equals $f(p,q)$.  An easy calculation shows that $j'(\ell)>0$ because $\ell>1/2$, and since $p$ and $\ell$ increase together, this gives $\frac{\partial f}{\partial p}>0$, thus proving~(1).

\begin{figure}[ht]
  \centering
  \includegraphics[width=0.7\linewidth]{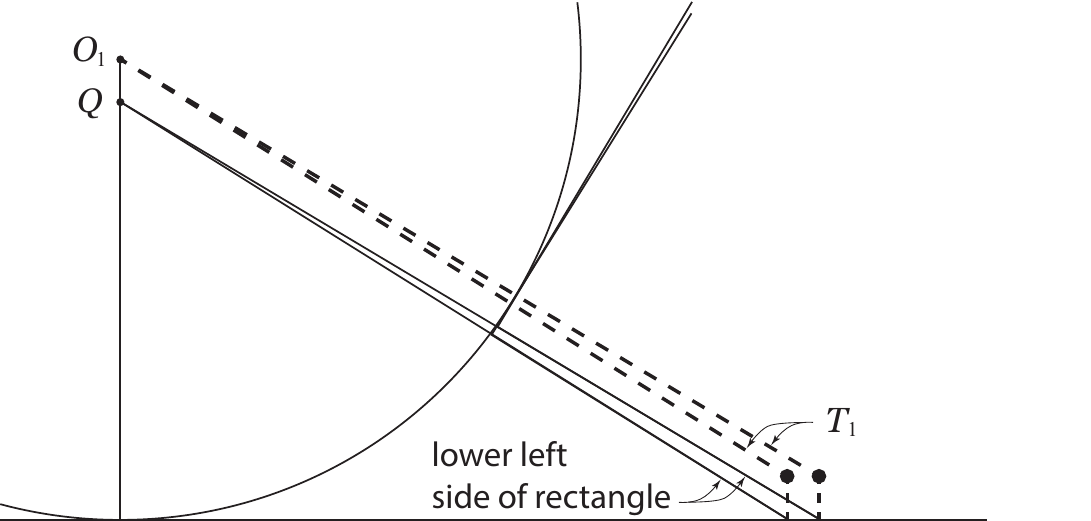}
  \caption{For Lemma~\ref{lem15}, geometry for the proof that $g(p,q)$ is a strictly decreasing function of~$p$.}
  \label{for15c}
\end{figure}

For (2), note that $g(p,q)$ is exactly $r$ less than the distance from $\oob$ to the lower left side of the rectangle.  That side is on a line through a point $Q$ below~$\ooa$.  If $p$ increases while $q$ stays fixed, then as shown in Figure~\ref{for15c}, the lower left side of the rectangle is still aligned with $Q$ because that side is always parallel to~$\tta$.  Therefore, as $p$ increases, the distance from $\oob$ to the lower left side of the rectangle decreases, thus proving~(2).

\begin{figure}[h]
  \centering
  \includegraphics[width=0.7\linewidth]{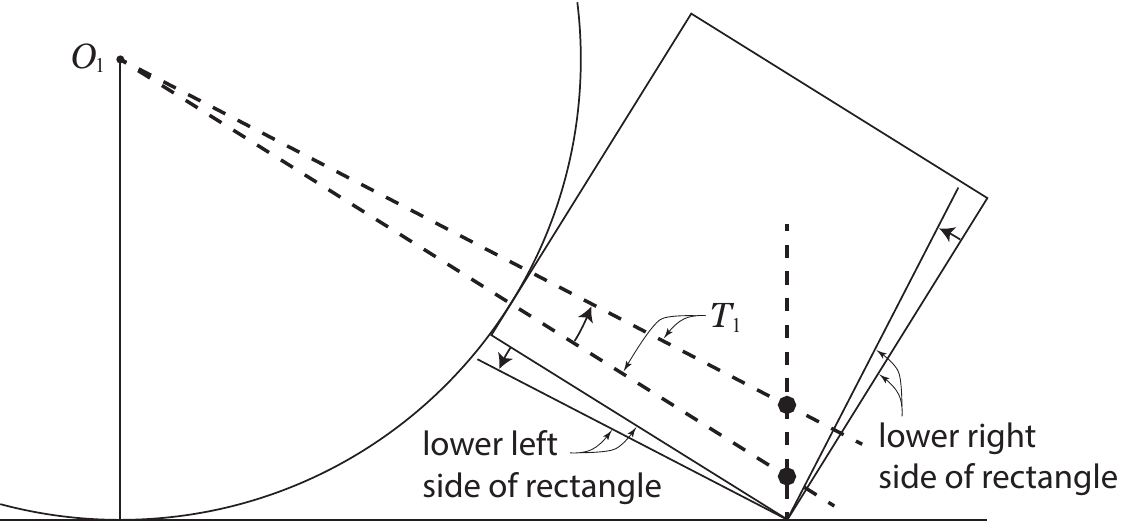}
  \caption{For Lemma~\ref{lem15}, geometry for proofs that $f(p,q)$ is strictly decreasing, and $g(p,q)$ is strictly increasing, as functions of~$q$.}
  \label{for15d}
\end{figure}

For (3) and (4), note that $f(p,q)$ is exactly $1$ less than the distance from $\ooa$ to the lower right side of the rectangle, and $g(p,q)$ is exactly $r$ less than the distance from $\oob$ to the lower left side of the rectangle.  If $p$ is fixed and $q$ increases, then as shown in Figure~\ref{for15d}, the lower left side of the square (which is parallel to~$\tta$) rotates away from $\oob$, and the lower right side (which is perpendicular to~$\tta$) rotates toward~$\ooa$.  This proves (3) and~(4), thus completing the proof.
\end{proof}

\begin{lem}\label{lem1618}
A minimal square cannot be in \con16 or \con18.
\end{lem}
\begin{proof}
In each of these configurations, $\tta$ intersects $\ttl$ above the intersection between $\ttl$ and~$\ttb$.  Therefore by Lemma~\ref{triangle}, $\sside(\aang)$ is strictly decreasing and therefore not minimal at the corresponding~$\aang$.
\end{proof}

If $r=1$, then \con19 is simply the reflection of \con6, so for \con19 we focus on the case $r\neq 1$.

\begin{lem}\label{lem19}
If $r\neq 1$, then a minimal square cannot be in \con19.
\end{lem}
\begin{proof}
Given a square at angle $\theta$ in \con19, we claim that the inscribed square in the reflected orientation, i.e.,\ at angle $\pi/2-\theta$, is smaller.  This fact can be seen by mapping the original square to a horizontal reflection that is positioned so that the image $P'$ of vertex $P$ ($=\vva$) is on~$\ccb$, as shown in Figure~\ref{for19}.  In the process, vertex $Q$ ($=\vvb$) is mapped to a point $Q'$ inside the disk bounded by $\cca$ for the following reason.  The acute angles $\phi$ made with the horizontal are the same for the line through $P$ and $Q'$ and the line through $P'$ and~$Q$, while circle $\cca$ is steeper than $\ccb$ at every height above $L$ between $0$ and~$1$.

\begin{figure}[ht]
  \centering
  \includegraphics[width=0.7\linewidth]{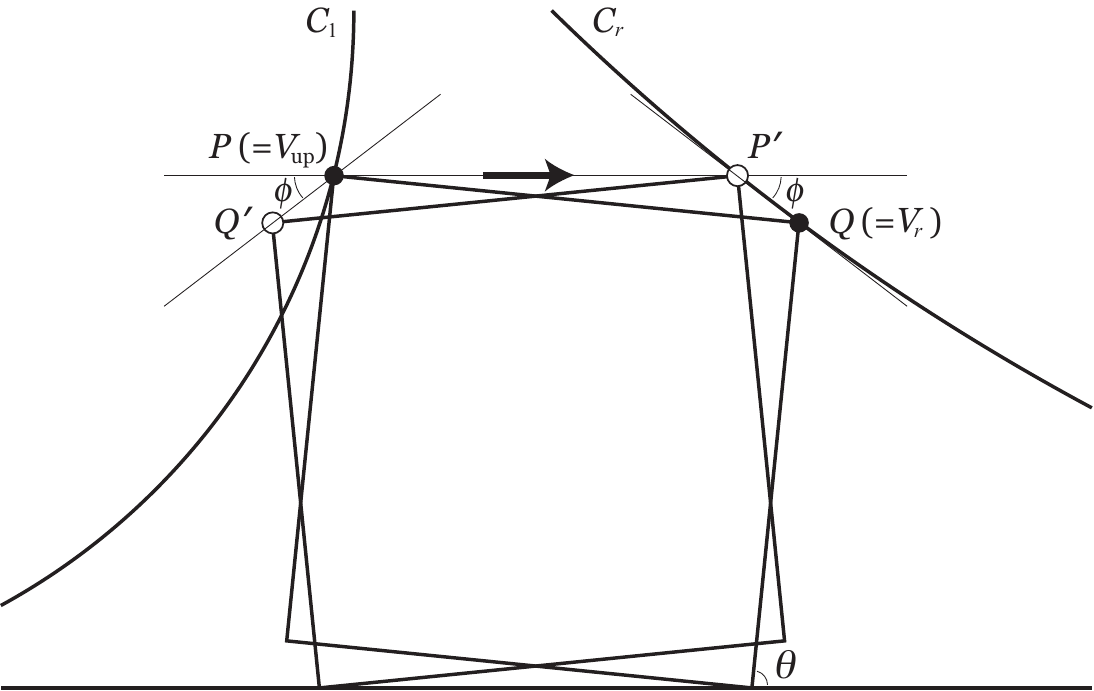}
  \caption{Illustration for Lemma~\ref{lem19} showing the horizontal reflection of the square that sends point $P$ on $\cca$ to a point $P'$ on~$\ccb$.  (Such a reflection is not necessarily about the midline of the square.)}
  \label{for19}
\end{figure}

Therefore, the inscribed square at angle $\pi/2-\theta$ must be smaller than the reflected square and thus the original square.
\end{proof}

To prove that a minimal square cannot be in \con8, \con9, or \con17, preliminary results and definitions are helpful.  

\begin{definition}
\emph{\con10$^+$} will refer to \con10 but will also allow $r\leq 1$.  \emph{\con8$^+$} will refer to the union of \con7, \con8, \con9, and \con10, and will also allow $r\leq 1$.
\end{definition}

\begin{rem}
A proof that a minimal square cannot be in \con8$^+$ constitutes a proof for \con8, \con9, and \con17 (as well as \con7, \con10, \con16, and \con18, for which proofs have already been given) because \con17 with $r=r_0\ge 1$ is equivalent to a version of \con8 with $r=1/r_0\le 1$.
\end{rem}

\begin{lem}\label{lessthanM}
The square in \con10$^+$ has side length less than or equal to~$M$, where
\[
M = \frac{2r}{r+\sqrt{8}\sqrt{r}+1},
\]
with equality if and only if $r=1$.
\end{lem}
\begin{proof}
Let $P$ be the point on line $L$ between $\cca$ and $\ccb$ such that the distance from $P$ to the closest point on $\cca$ (along the line through $P$ and~$\ooa$) equals the distance from $P$ to the closest point on~$\ccb$.  The line segment from $P$ to $\ooa$ can be viewed as the hypotenuse of a right triangle, as can the line segment from $P$ to~$\oob$, so the equality of distance can be written using the Pythagorean theorem as follows.  Let $p$ be the distance from $P$ to the intersection of $\lli$ and~$\cca$.  Because $2\sqrt{r}$ is the distance from the intersection of $\lli$ and $\cca$ to the intersection of $\lli$ and~$\ccb$ (see Figure~\ref{equations}), we have
\[
\sqrt{1+p^2} - 1 = \sqrt{r^2+(2\sqrt{r}-p)^2} - r.
\]
By basic geometry, the square in \con10$^+$ must have $\vvdn$ to the left of~$P$, unless $r=1$, in which case $\vvdn=P$.  Therefore the square has side length less than or equal to $\sqrt{1+p^2} - 1$, with equality exactly if and only if $r=1$.

It is straightforward to confirm, by substituting the following into the equation above, and by noting that this is the desired solution because it lies in $(0,2\sqrt{r})$ for $r>0$, that
\[
p = \frac{\sqrt{8}r + 2\sqrt{r}}{r+\sqrt{8}\sqrt{r}+1},
\]
and then that
\[
\sqrt{1+p^2} - 1 = \frac{2r}{r+\sqrt{8}\sqrt{r}+1},
\]
which is $M$ in the statement of the lemma.
\end{proof}

\begin{lem}\label{lem17a}
For a square in \con17, if the distance from $L$ to the midpoint of the upper right side of the square is less than or equal to $r-r/\sqrt{2}$, then the square is not minimal.
\end{lem}
\begin{proof}
Given such a square in \con17 as shown in Figure~\ref{midlinereflection}, we claim that reflecting the square about its midline --- the vertical line through its center point --- will cause it to overlap both $\cca$ and~$\ccb$.  This will prove the lemma because an inscribed square in the same orientation as the reflected square is thus smaller than the original square.  We prove the claim by showing that reflecting the square takes vertex $P$ (= $\vvup$) as shown in Figure~\ref{midlinereflection} to vertex $P'$ in the interior of the disk with boundary $\ccb$, and takes $Q$ (= $\vvb$) to vertex $Q'$ such that either $Q'$ or part of the left side of the reflected square is in the interior of the disk with boundary~$\cca$.

\begin{figure}[ht]
  \centering
  \includegraphics[width=0.85\linewidth]{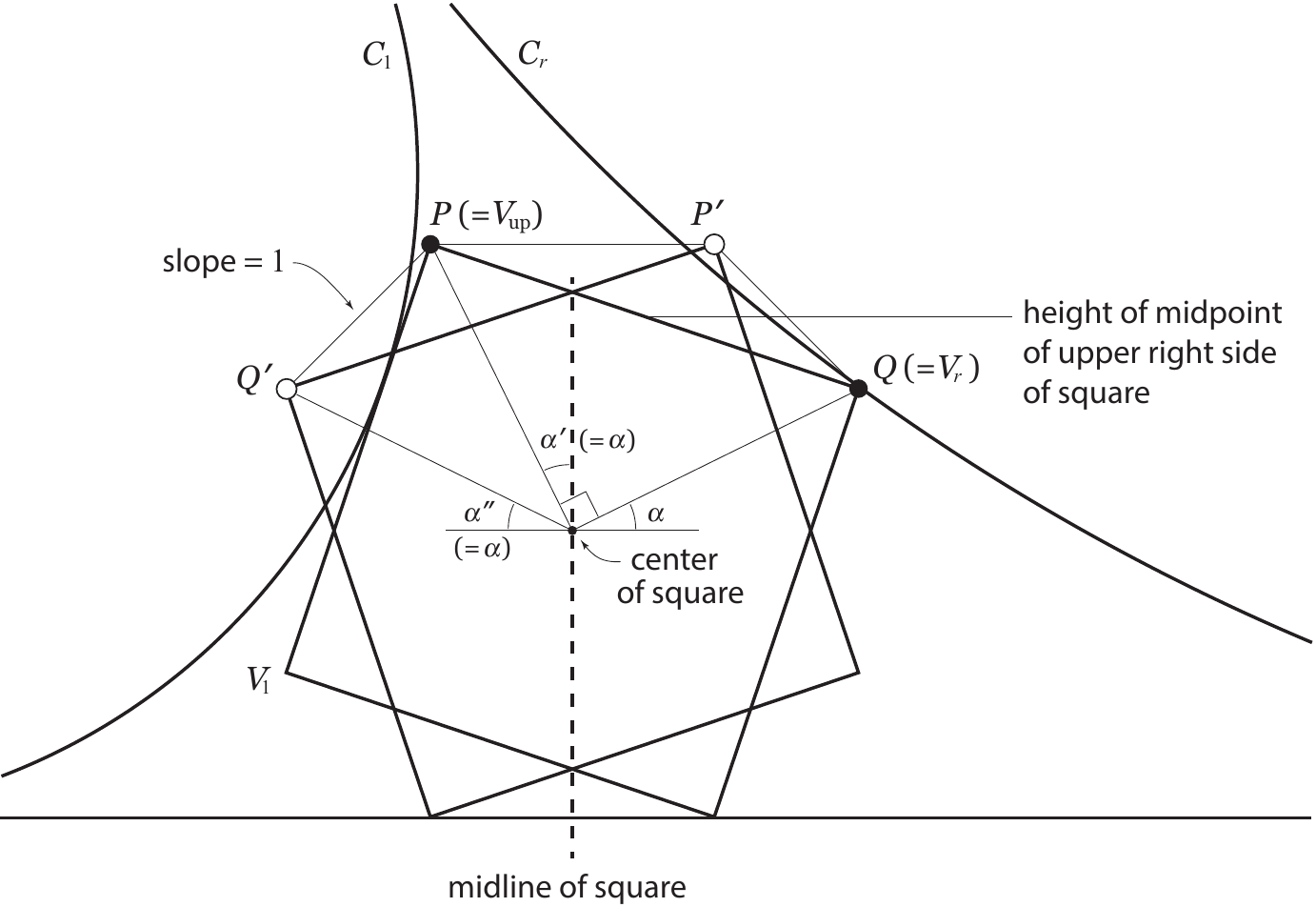}
  \caption{For Lemma~\ref{lem17a}, the horizontal reflection of the square about its midline.}
  \label{midlinereflection}
\end{figure}

First, the line through $P$ and $Q'$ has slope~$1$, and the line through $Q$ and $P'$ has slope $-1$, by the following reasoning:  In Figure~\ref{midlinereflection}, the fact that $\alpha'=\alpha$ is clear because of the right angle shown, and $\alpha''=\alpha$ because of the symmetry of the reflection.  The line segment between $P$ and $Q'$ is seen to be the base of an isoceles triangle since $P$ and $Q'$ are the same distance from the center of the square.  Because $\alpha'=\alpha''$, the isoceles triangle is symmetric about a line of slope~$-1$ through the center of the square. Therefore, the base of the triangle, and thus the line through $P$ and $Q'$, has slope~$1$.  Similar reasoning shows that the line through $Q$ and $P'$ has slope~$-1$.

Next, we claim that the secant line segment intersecting $\ccb$ at $Q$ and the point nearest horizontally to $P'$ has slope shallower than $-1$.  (This secant would lie almost exactly along $\ccb$ in Figure~\ref{midlinereflection}, and would not be distinguishable in the figure.)  To prove the claim, we note that the midpoint of that secant is at the same height as the upper right side of the square, no higher than $r-r/\sqrt{2}$ by hypothesis, because the secant's endpoints are at the heights of vertices $P$ and~$Q$.  A secant has as a perpendicular bisector a radial line segment of~$\ccb$; such a radial line segment for this secant must reach below a height of $r-r/\sqrt{2}$ since it passes through the secant's midpoint.  Thus, this radial line segment must have slope steeper than~$1$, because at slope $1$ it would reach down only to height $r-r/\sqrt{2}$ with its length $r$ and top endpoint at height~$r$.  The secant, which is perpendicular to it, must therefore have slope shallower than~$-1$.

Therefore, $P'$ is in the interior of the disk with boundary~$\ccb$.

To show the overlap with $\cca$, we first note that $Q'$ is at height greater than $1-1/\sqrt{2}$.  This is because $\vvb$ of a square in \con17 must be higher than that in \con9 or \con10 (whichever applies, depending on whether $r=1$) by the basic geometry of the counterclockwise rotation toward \con17.  Among these three configurations, \con9 with its $r=1$ has the lowest $\vvb$, at height $1-1/\sqrt{2}$ since $\aang=\pi/4$.

Now, $Q'$ lies on the line of slope $1$ through~$P$, a line whose lower intersection with $\cca$ is at height less than $1-1/\sqrt{2}$ because $1-1/\sqrt{2}$ is where the tangent to $\cca$ has slope~$1$.  Therefore, $Q'$ is on the line of slope $1$ through~$P$ on either the segment from $P$ to~$\cca$ or the segment that lies in the interior of the disk with boundary $\cca$.  In the latter case, the proof of the lemma is complete.  The former case can only occur if the original square is so large that $Q'$ is above a portion of the top half of~$\cca$, because $Q'$ is straight above~$\vva$.  In that case, the left side of the reflected square must intersect the interior of the disk with boundary $\cca$, completing the proof of the lemma, because the reflected square is also a reflection of the original square about a \textit{horizontal} line through its center point.  Since the left side of the original square is tangent to $\cca$, and the horizontal line of reflection is below the center of~$\cca$, the reflected left side must overlap $\cca$ and thus intersect the interior of the disk with boundary~$\cca$.
\end{proof}

\begin{cor}\label{cor17a}
If $r\geq 3$, then a minimal square cannot be in \con17.
\end{cor}
\begin{proof}
Suppose for contradiction that $r\geq 3$ and a minimal square is in \con17.  Then the side length, $s$, must be less than or equal to the side lengths of all squares in all other configurations, including \con10 which is addressed in Lemma~\ref{lessthanM}.  In particular, we must have $s\leq M$ from Lemma~\ref{lessthanM}.  Let $h$ be the distance from $L$ to the midpoint of the upper right side of the square. We know that $h$ can be no greater than $(\sqrt{5}/2)s$, because this is the highest that a midpoint of a side can possibly be, occurring if the midpoint is straight above~$\vvdn$.  In summary, $h\leq (\sqrt{5}/2)s \leq (\sqrt{5}/2)M$.  Therefore, $h\leq r-r/\sqrt{2}$ if
\[
\frac{\sqrt{5}}{2} \left( \frac{2r}{r+\sqrt{8}\sqrt{r}+1} \right) \leq r - \frac{r}{\sqrt{2}},
\]
which is true if $r\geq 3$.  However, Lemma~\ref{lem17a} then implies that the square is not minimal.  This gives a contradiction.
\end{proof}

\begin{definition}
For a family $\{J(t)\}_{t\in I}$ of lines or line segments, where $I$ is an interval in~$\R$, the \emph{pivot} at a given $t\in I$ is the point about which the line or line segment is pivoting at~$t$, if such a point exists.  In other words, for a given $t\in I$, if there exists $\epsilon>0$ such that $u\in I \cap (t-\epsilon, t+\epsilon)$ implies that $J(u) \cap J(t)$ consists of a single point, then the pivot, $P$, at~$t$ is defined by
\[
P = \lim_{u\to t}  J(u) \cap J(t)
\]
if the limit exists.  (The limit is one-sided if $t$ is an endpoint of~$I$.)
\end{definition}

\begin{lem}\label{pivot}
Let $L_\phi$ be the line through the origin having angle $\phi\in (0,\pi/2)$ with the positive $x$-axis as in Figure~\ref{jknotation}, define an interval $I=(\pi/2-\phi,\pi/2)$, and fix $\ell>0$.  Let $\{J(\beta)\}_{\beta\in I}$ be the family of line segments of length $\ell$ with endpoints on the $x$-axis and $L_\phi$, and for which $J(\beta)$ intersects the $x$-axis at angle~$\beta$ as in Figure~\ref{jknotation}.  Then the following hold.

\begin{figure}[h]
  \centering
  \includegraphics[width=0.7\linewidth]{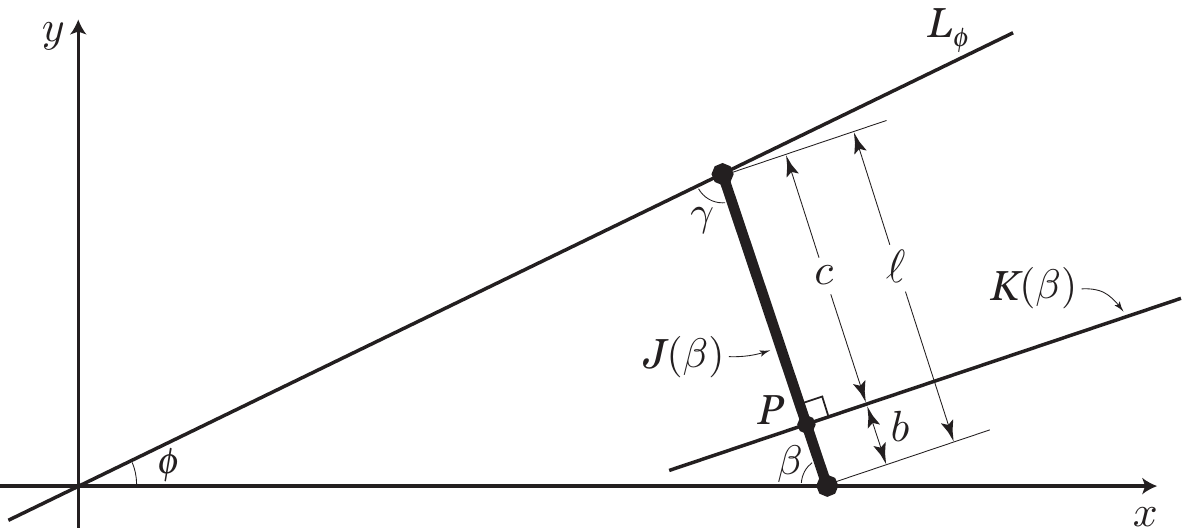}
  \caption{Notation for Lemma~\ref{pivot}.}
  \label{jknotation}
 \end{figure}
 
\begin{itemize}
\item[]
\begin{itemize}
\item[(i)]  For $\{J(\beta)\}_{\beta\in I}$, the position of the pivot $P$ at $\beta$ is determined by
\[
b\,(\cot\beta) = c\,(\cot\gamma),
\]
where $b$ and $c$ are the distances along $J(\beta)$ from $P$ to the $x$-axis and $L_\phi$, respectively, and $\gamma=\pi-\phi-\beta$ is the angle between $J(\beta)$ and $L_\phi$.
\item[(ii)]  Let $\{K(\beta)\}_{\beta\in I}$ be the family of lines such that $K(\beta)$ is the line perpendicular to $J(\beta)$ through $J(\beta)$'s pivot.  Then for $\{K(\beta)\}_{\beta\in I}$, the pivot at $\beta$ has $y$-coordinate
\[
\ell\left(\frac{\cos\gamma + \cos(\gamma-\beta)\cos\beta}{\sin(\gamma+\beta)}\right).
\]
\end{itemize}
\end{itemize}

\end{lem}

\begin{proof}
Given such $L_\phi$, $\ell$, $\{J(\beta)\}_{\beta\in I}$, and $\{K(\beta)\}_{\beta\in I}$, fix a value $\beta_0\in I$.

To prove (i), let $P$ be the pivot at $\beta_0$ for $\{J(\beta)\}_{\beta\in I}$, and let $b$ and $c$ be the distances along $J(\beta_0)$ from $P$ to the $x$-axis and $L_\phi$, respectively.  Define a family $\{H(\beta)\}_{\beta\in(0,\pi/2)}$ of variable-length line segments as those for which $H(\beta)$ is the line segment through $P$ with endpoints on the $x$-axis and $L_\phi$, and that intersects the $x$-axis at angle $\beta$; e.g.,\ $H(\beta_0)=J(\beta_0)$.  It follows that $d(\text{length of $H(\beta)$})/d\beta = 0$ at $\beta=\beta_0$ because $H(\beta)$ and $J(\beta)$ have the same pivot at~$\beta_0$, so the instantaneous rates of change of location of their respective endpoints (on the $x$-axis and~$L_\phi$) are the same.  Let $B(\beta)$ and $C(\beta)$ be the distances along $H(\beta)$ from $P$ to the $x$-axis and $L_\phi$, respectively, so $B'(\beta_0)+C'(\beta_0)=0$.  Basic trigonometry shows that $B(\beta)=h_B/\sin\beta$ and $C(\beta)=h_C/\sin\gamma$, where $\gamma=\pi-\phi-\beta$, and $h_B$ and $h_C$ are the (shortest) distances from $P$ to the $x$-axis and $L_\phi$, respectively.  Putting all of this together, (i)~follows because $b=B(\beta_0)$ and $c=C(\beta_0)$.

For (ii), we note that $K(\beta)$ is the line
\begin{align*}
y &= \cot\beta\left(x -\left(\frac{\ell\sin\gamma}{\sin\phi} - b\cos\beta\right)\right) + b\sin\beta, \\
\text{i.e.,\ \ \ \ } y &= (\cot\beta)x - \ell\left(\frac{\sin(\phi + 2\beta)}{\sin\phi \sin\beta}\right),
\end{align*}
where the second form uses $\gamma = \pi-\phi-\beta$, as well as a substitution for $b$ stemming from (i) which gives $b (\cot\beta) = (\ell-b) (\cot\gamma)$.

For all $\beta$, the slope of $K(\beta)$ is greater than zero, so the pivot for $\{K(\beta)\}_{\beta\in I}$ at $\beta_0$ can be found as follows.  For any given~$y$, define the function $f_y(\beta)$ giving the value of $x$ such that $(x,y)$ is on the line $K(\beta)$:
\[
f_y(\beta) = (\tan\beta)y + \ell\left(\frac{\sin(\phi+2\beta)}{\sin\phi\cos\beta}\right).
\]
Then the pivot at $\beta_0$ must have $y$-coordinate such that $f_y'(\beta_0)=0$.  Solving this equation for $y$ gives~(ii).
\end{proof}

\begin{lem}\label{lem8917}
A minimal square cannot be in \con8, \con9, or \con17.
\end{lem}
\begin{proof}
Suppose for contradiction that a minimal square exists in \con8$^+$.  Noting that the extended version of \con8 with $r=r_0\le 1$ is equivalent to \con17 with $r=1/r_0$, we may apply Corollary~\ref{cor17a}, concluding that $r>1/3$.

Let $I=[\aalft_1,\aalft_2]$ be the set of angles between $L$ and the lower left sides of the squares in \con8$^+$, and let $\aalftz\in I$ be the maximum such angle of a minimal square in \con8$^+$.  Then we know that $\aalftz\neq\aalft_2$, and that $\aalftz\neq\aalft_1$ unless $r=1$, because a minimal square cannot be in \con7, \con10, \con16, or \con18.  Let $\ell$ be the side length of the minimal square, and let $\{G(\aalft)\}_{\aalft\in I}$ be the family of line segments of length $\ell$ with endpoints on $L$ and $C_1$, for which $G(\aalft)$ intersects $L$ at angle~$\aalft$.  Let $m(\aalft)$ be the distance from $\oob$ to $G(\aalft)$.  Then $m'(\aalftz)=0$ because either $m(\aalftz)$ is a minimal value on an open interval, or the square for $\aalftz$ is in \con9 so the shortest distance from $\oob$ to $G(\aalftz)$ is the distance from $\oob$ to $G(\aalftz)$'s endpoint at $L$, which is the pivot for $\{G(\aalft)\}_{\aalft\in I}$ at~$\aalftz$.
 
Let $\{H(\aalft)\}_{\aalft\in I}$ be the family of lines such that $H(\aalft)$ is perpendicular to $G(\aalft)$ through $G(\aalft)$'s pivot.  Let $h(\aalft)$ be the vertical distance from $\oob$ to~$H(\aalft)$, but considered negative if $H(\aalft)$ is below~$\oob$.  The sign of $m'(\aalft)$ is the same as the sign of~$h(\aalft)$, because $h(\aalft)>0$ means that $G(\aalft)$ is moving away from~$\oob$ as $\aalft$ increases, and vice versa.  Therefore, $h(\aalftz)=0$.

Let $L_0$ be the line tangent to $\cca$ at $G(\aalftz)$'s endpoint on~$\cca$.  It follows from this tangency of $L_0$ and~$\cca$ that the pivot at $\aalftz$ for $\{G(\aalft)\}_{\aalft\in I}$ is the same as the pivot at $\aalftz$ for the family $\{J(\aalft)\}_{\aalft\in I}$ of line segments of length $\ell$ with endpoints on $L$ and $L_0$ with $J(\aalft)$ intersecting $L$ at angle~$\aalft$.  Intuitively, this match of the pivot is because $J(\aalft)$ acts like $G(\aalft)$ near~$\aalftz$, and can be seen formally by standard $\epsilon$-$\delta$ reasoning.  It also follows that the pivot $P$ at $\aalftz$ for $\{H(\aalft)\}_{\aalft\in I}$ is the same as the pivot at $\aalftz$ for the family $\{K(\aalft)\}_{\aalft\in I}$ of lines such that $K(\aalft)$ is perpendicular to $J(\aalft)$ through $J(\aalft)$'s pivot.

We now apply Lemma~\ref{pivot}(ii) to $\{K(\aalft)\}_{\aalft\in I}$ in order to show that the distance from $L$ to $P$ is less than~$r$.  This distance, the ``$y$-coordinate'' of~$P$, is thus given by
\[
\yp = \ell\left(\frac{\cos\gamma + \cos(\gamma-\aalft)\cos\aalft}{\sin(\gamma+\aalft)}\right),
\]
where $\gamma$ is the angle that $J(\aalft)$ makes with~$L_0$.  We claim that $\yp\leq \ell\sqrt{2}$, and begin the proof by noting that for angle $\phi$ between $L$ and $L_0$ we have $\phi+\gamma+\aalft=\pi$, along with $\phi,\aalft\in(0,\pi/2)$, $\gamma\in(0,\pi/2]$, and $\gamma+\aalft\in(\pi/2,\pi)$.  Let $v=(\gamma+\aalft)/2$ and $w=(\gamma-\aalft)/2$, and let
\[
f(v,w) = \ell\left(\frac{\cos(v+w) + \cos(2w)\cos(v-w)}{\sin(2v)}\right)
\]
so that $\yp\leq\ell\sqrt{2}$ can be proved by showing that $f(v,w)\leq\sqrt{2}$ on the domain defined by $v+w\leq\pi/2$, $v-w<\pi/2$, and $v\geq\pi/4$.  Although $v$ cannot take the value $\pi/4$ in the geometric interpretation, this extension of the domain of $f$ to include $v=\pi/4$ facilitates phrasing in the proof of a bound.

Straightforward computations show that
\[
\frac{\partial f}{\partial v} = \frac{-\cos^3\aalft - 3\cos\gamma\cos\aalft}{\sin^2(\gamma+\aalft)} <0
\]
on the domain, so the maximum occurs on the border $v=\pi/4$.  Let
\[
g(w) = f\left(\frac{\pi}{4},w\right) = \cos\left(w+\frac{\pi}{4}\right) + \cos(2w)\cos\left(w-\frac{\pi}{4}\right).
\]
The maximum value of $g(w)$ occurs at $w=0$, because
\[
g'(w) = \sqrt{8}\left(\cos^3 w - \cos w - \cos^2 w \sin w\right)
\]
and $w\in (-\pi/4,\pi/4)$ on the domain, so $g(w)$ increases then decreases, and moreover, $g'(w)=0$ only at $w=0$.  Therefore, we have $f(v,w)\leq f(\pi/4,0) = \sqrt{2}$ for all $(v,w)$ in the domain of~$f$, proving that $\yp\leq\ell\sqrt{2}$.

To complete the proof that the distance from $L$ to $P$ is less than~$r$, we note that
 because $\ell$ is the minimum side length of a square, Lemma~\ref{lessthanM} implies that
\[
\yp\leq\ell\sqrt{2}\leq \left(\frac{2r}{r+\sqrt{8}\sqrt{r}+1}\right)\sqrt{2}.
\]
Since we know from Corollary~\ref{cor17a} that $r>1/3$, the fact that $\yp<r$ easily follows.

Now, because $H(\aalftz)$ has positive slope, and pivot $P$ on $H(\aalftz)$ is at a distance less than $r$ from~$L$, $P$~is below and to the left of~$\oob$.  Recalling that $h(\aalft)$ is the (signed) vertical distance from $\oob$ to~$H(\aalft)$, the fact that pivot $P$ is to the left of $\oob$ implies that $h'(\aalftz)<0$.  We already know that $h(\aalftz)=0$, so there exists $\epsilon>0$ such that $\aalft\in (\aalftz,\aalftz+\epsilon)$ implies $h(\aalft)<0$.  However, since $h(\aalft)$ and $m'(\aalft)$ have the same sign, this means $m'(\aalft)<0$ on $(\aalftz,\aalftz+\epsilon)$, contradicting $m(\aalftz)$ being minimal.
\end{proof}

\begin{prop}\label{prop6only}
A minimal square occurs only in \con6, and possibly in \con19, and must have $\vva$ lower than~$\ooa$.  \con19 has a minimal square if and only if $r=1$.
\end{prop}
\begin{proof}
This result follows from the lemmas ruling out all other configurations, and the fact that \con19 is the reflection of \con6.
\end{proof}

\section{Equations for minimum side length.}\label{polynomial_section} 
In this section we derive equations toward the pursuit of the minimum side length of an inscribed square, knowing from Proposition~\ref{prop6only} that this minimum occurs in \con6 and has $\vva$ lower than~$\ooa$.  We continue using here the notation introduced at the beginnings of Sections 2 and~3.

\begin{lem}\label{lem6monotonic}
In \con6 with $\vva$ lower than~$\ooa$, $s(\theta)$ decreases then increases, both strictly monotonically, as a function of~$\aang$.  Thus, there is a unique minimal square in~\con6.
\end{lem}
\begin{proof}
As $\theta$ increases in \con6 with $\vva$ lower than~$\ooa$, $\ttl$~moves to the right and $\vva$ moves down, so the intersection between $\tta$ (which has negative slope) and $\ttl$ moves down.  Simultaneously, $\vvb$ moves up, so the intersection between $\ttb$ and $\ttl$ moves up.  Therefore, as $\theta$ increases, the triangle formed by $\ttl$, $\tta$, and $\ttb$ switches from the right side of $\ttl$ to the left side of $\ttl$ at a unique value, when $\tta$ and $\ttb$ intersect $\ttl$ at the same point.  By Lemma~\ref{triangle}, $s(\theta)$~decreases then increases, both strictly monotonically.

The uniqueness of the minimal square in \con6 follows from this and Proposition~\ref{prop6only}.
\end{proof}

For the equations used in seeking the minimum side length of an inscribed square, the parameter representing orientation will be the distance, $\xt$, from the line $\lli$ to the upper left vertex $\vva$ of the square, as shown in Figure~\ref{equations}.  It is not difficult to see that $\xt$ is related strictly monotonically to~$\theta$.

\begin{figure}[ht]
  \centering
  \includegraphics[width=\linewidth]{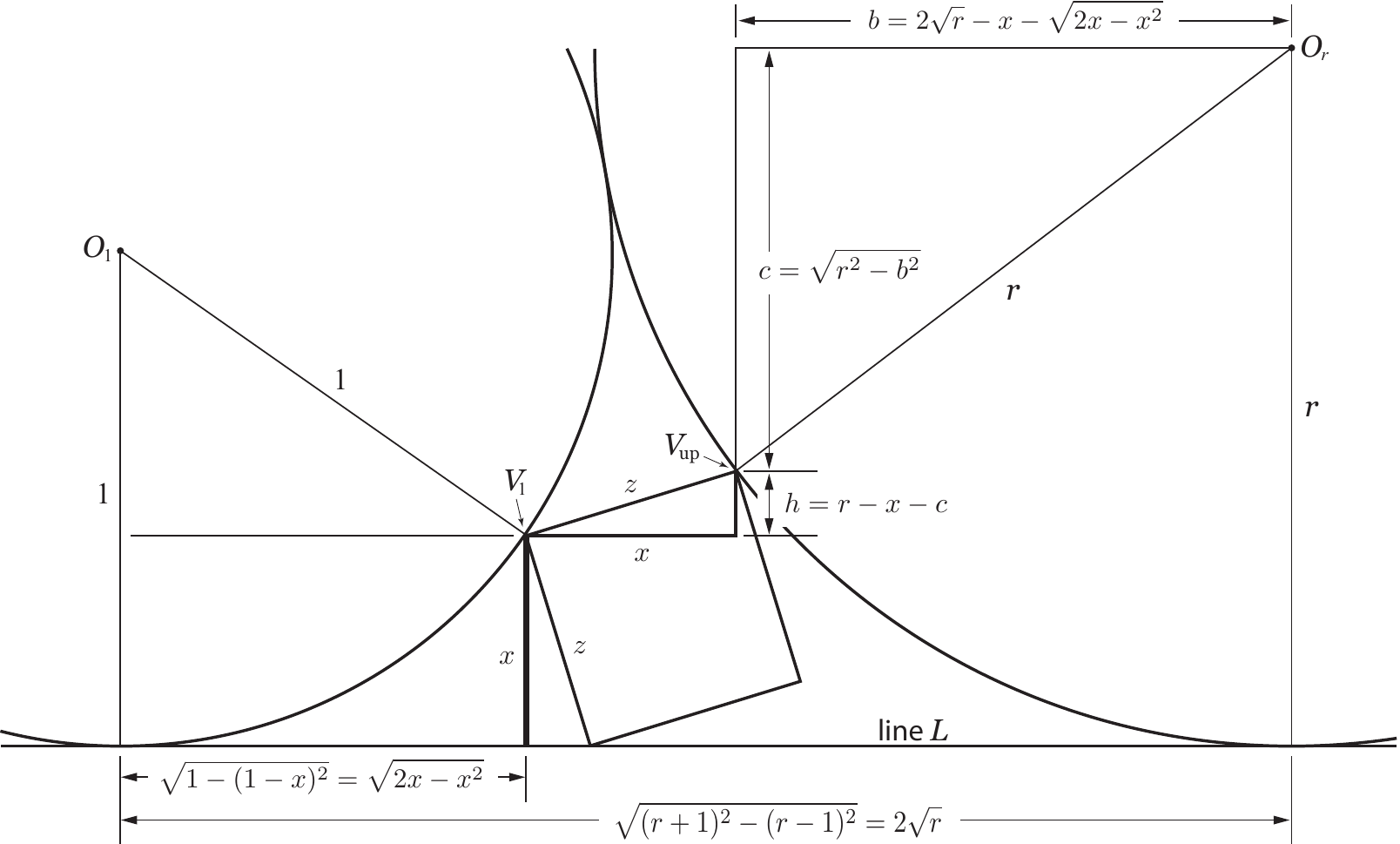}
  \caption{Equations used to compute the side length of a square in \con6, given distance $\xt$ from the line $\lli$ to the upper left vertex $\vva$ of the square.  Here the side length is denoted~$z$.}
  \label{equations}
\end{figure}

\begin{prop}\label{minsidelength}
The minimum side length, $\mu(r)$, of an inscribed square is the minimum value of the function
\begin{equation}\label{side_length_equation}
z(\xt) = \left[\xt^2 + \left(r - \xt - \sqrt{r^2 - \left(2\sqrt{r} - \xt - \sqrt{2\xt - \xt^2}\right)^2}\right)^2 \right]^\frac{1}{2}
\end{equation}
on the interval $(1-1/\sqrt{2},1)$.  The minimum occurs at a unique point, denoted~$\xt_m$.  Moreover, the function $z$ is strictly decreasing for $\xt<\xt_m$ and strictly increasing for $\xt>\xt_m$.
\end{prop}
\begin{proof}
Let $(\xt_1,\xt_2)$ be the interval of values of the distance $\xt$ from $\lli$ to $\vva$ such that the associated square is in \con6 with $\vva$ lower than~$\ooa$. We know from Proposition~\ref{prop6only} that the minimum side length corresponds to some $\xt\in (\xt_1,\xt_2)$.

If $\xt\in (\xt_1,\xt_2)$, then as illustrated in Figure~\ref{equations}, congruent triangles using $\xt$ and $z$ show that  $\vvup$ is at a horizontal distance $\xt$ from~$\vva$.  Therefore, for $\xt\in(\xt_1,\xt_2)$,
\[
z = \left(\xt^2+h^2 \right)^\frac{1}{2},
\]
where $h$ is the vertical distance from $\vva$ to~$\vvup$.  Thus by the equations in Figure~\ref{equations}, the side length is given as a function of $\xt\in(\xt_1,\xt_2)$ by $z(\xt)$ in equation~\eqref{side_length_equation}.  In addition, it follows from Lemma~\ref{lem6monotonic} that on $(\xt_1,\xt_2)$ the minimum occurs at a unique point, $\xt_m$, and that $z$ is strictly decreasing for $\xt<\xt_m$ and strictly increasing for $\xt>\xt_m$.

However, the stated domain $(1-1/\sqrt{2},1)$ extends beyond $(\xt_1,\xt_2)$; certainly, $\xt_2<1$ by definition.  Toward the claim that $\xt_1> 1-1/\sqrt{2}$, we suppose for contradiction that $\xt=1-1/\sqrt{2}$ is possible.  Because $\vva$ is distance $x$ from~$\lli$, $\xt=1-1/\sqrt{2}$ occurs exactly when the line through $\ooa$ and $\vva$ has slope~$-1$.  In that case, $1-1/\sqrt{2}$ is also the horizontal distance from $\vva$ to a vertical line tangent to $\cca$ on the right side.  Because $\vvup$ is horizontal distance $x$ ($=1-1/\sqrt{2}$) from $\vva$, $\vvup$ is on that vertical line.  This contradicts the fact that $\vvup$ is on $\ccb$ which requires $\vvup$ to be to the right of that vertical line.  Noting that $x< 1-1/\sqrt{2}$ would result in a similar contradiction, we conclude that $\xt_1> 1-1/\sqrt{2}$.

Now considering $z$ in equation~\eqref{side_length_equation} as a function of $\xt\in (1-1/\sqrt{2},1)$, we know on $(\xt_1,\xt_2)$ that $z$ represents the side length of an inscribed square, but outside $(\xt_1,\xt_2)$ we only have equation~\eqref{side_length_equation}.  It is helpful to set up geometric interpretations for the cases $\xt\in (1-1/\sqrt{2},\xt_1)$ and $\xt\in (\xt_2,1)$ in order to complete the proof by showing that $z$ strictly decreases when $\xt\in (1-1/\sqrt{2},\xt_1)$ and strictly increases when $\xt\in (\xt_2,1)$.

\begin{figure}[h]
  \centering
  \includegraphics[width=\linewidth]{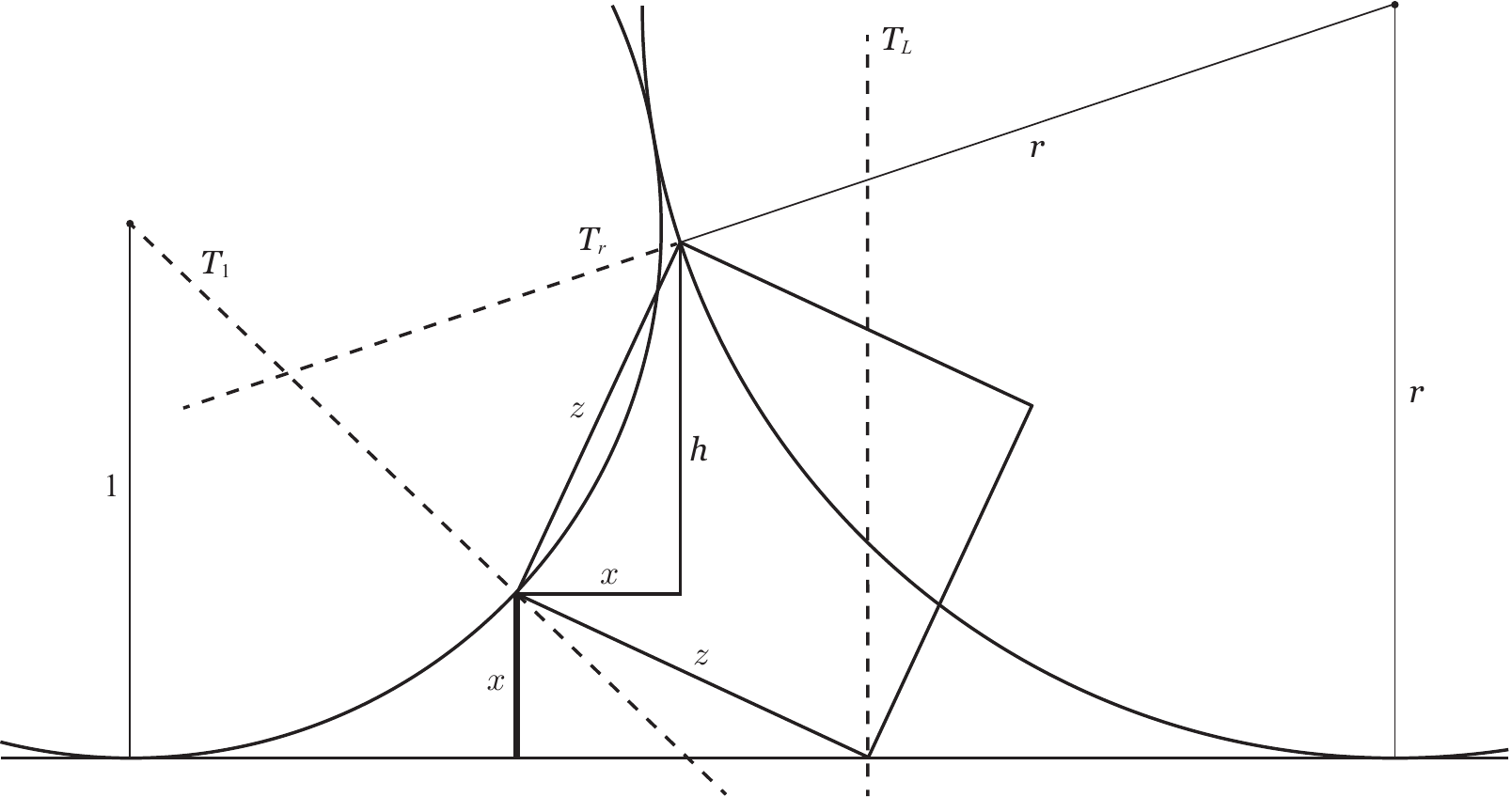}
  \caption{Geometry showing that $z$ given by \eqref{side_length_equation} is not a minimum when $\xt<\xt_1$.}
  \label{smallx}
\end{figure}

\begin{figure}[h]
  \centering
  \includegraphics[width=\linewidth]{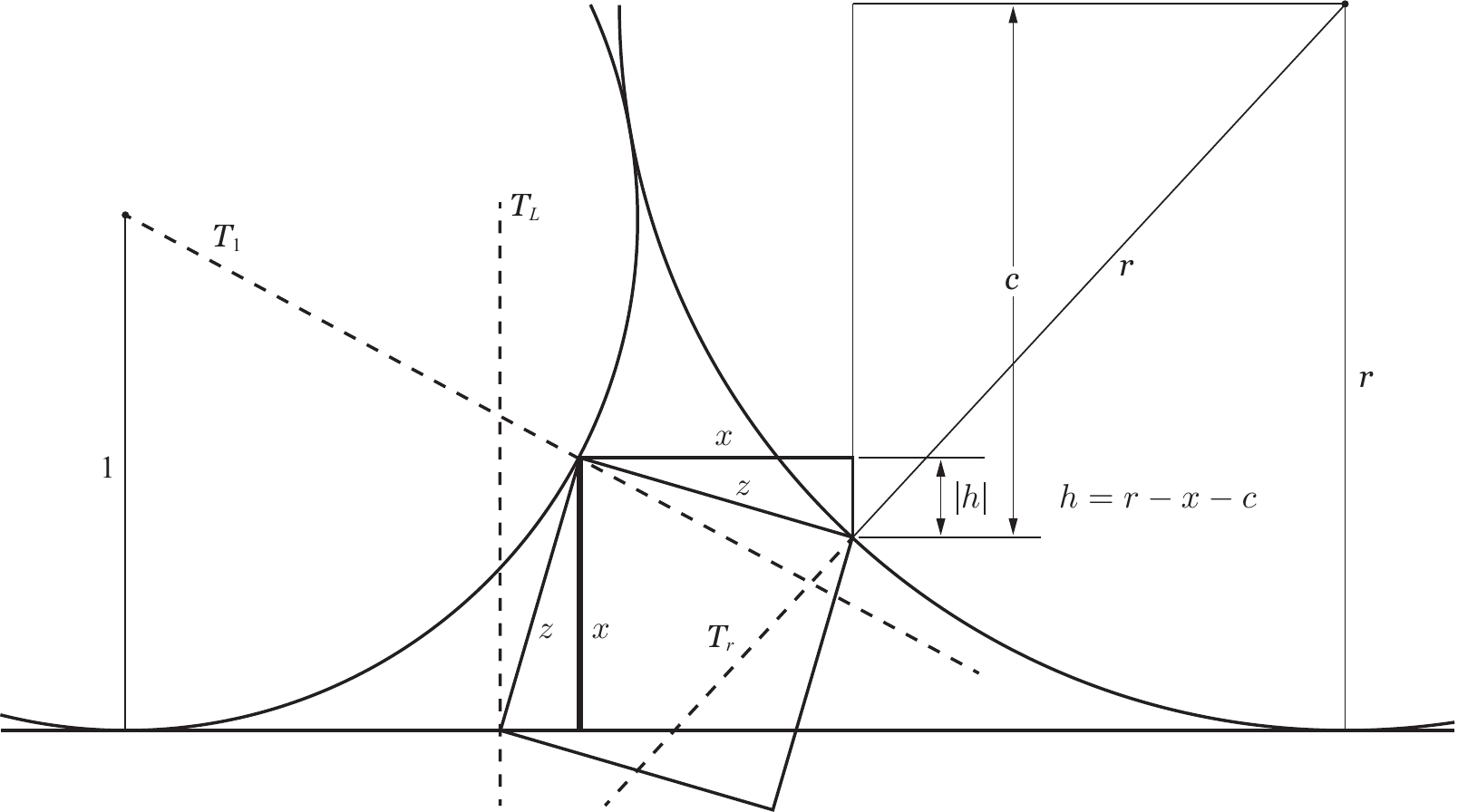}
  \caption{Geometry showing that $z$ given by \eqref{side_length_equation} is not a minimum when $\xt>\xt_2$.}
  \label{largex}
\end{figure}

Figures \ref{smallx} and~\ref{largex} illustrate geometric interpretations of $z$ as given by equation~\eqref{side_length_equation} when $\xt$ belongs to the intervals $(1-1/\sqrt{2},\xt_1)$ and $(\xt_2,1)$, respectively.  Here, the three consecutive vertices still lie on the line and the two circles, although the squares are no longer inscribed.  In both cases, the same equations as in Figure~\ref{equations} still apply, and lead to \eqref{side_length_equation}.  The fact that $z$ strictly decreases when $\xt\in (1-1/\sqrt{2},\xt_1)$ and strictly increases when $\xt\in (\xt_2,1)$ can been seen by the same reasoning as in Lemma~\ref{lem6monotonic}, this time using $\tta$, $\ttb$, and $\ttl$ defined as before except that the ``point of contact'' is specifically that with the relevant vertex, as shown in Figures \ref{smallx} and~\ref{largex}.  The reasoning then uses a modified form of Lemma~\ref{triangle} that applies to these squares with consecutive vertices on the line and circles, instead of to inscribed squares, and whose proof is analogous to that of Lemma~\ref{triangle}.
\end{proof}

\begin{rem}[approximating $\mu(r)$]
From Proposition~\ref{minsidelength}, one can deduce an algorithm for computing an approximation of $\mu(r)$ given the radius $r$. Indeed, it suffices for this purpose to minimize the function $z$, which can be achieved by applying root-finding methods to $z'$.
\end{rem}

\begin{prop}\label{polyequation}
In the result of Proposition~\ref{minsidelength}, i.e.,\ that the minimum side length is given by $\mu(r) = z(x_m)$,
the number $\xt_m$ is a root of the 10th degree polynomial
\begin{equation}\label{poly_equation}
\fb(\fe\fh + \ff\fg - 2\fc\fd)^2 - (\fc^2\fb + \fd^2 - \fe\fb\fg - \ff\fh)^2,
\end{equation}
where, letting $\rt=\sqrt{r}$,
\begin{align*}
\fe(\xt) &= -2\xt + 4\rt \\
\ff(\xt) &= (4\rt - 2)\xt + \rt^4 - 4\rt^2 \\
\fc(\xt) &= (6\rt - 3)\xt + \rt^4 - 2\rt^3 - 3\rt^2 \\
\fb(\xt) &= - \xt^2 +2\xt \\
\fd(\xt) &= 4\xt^3 - (2\rt^2 + 6\rt + 7)\xt^2 + (2\rt^3 + 3\rt^2 + 10\rt)\xt - 2\rt^3 \\
\fg(\xt) &= 8\xt^3 + (-4\rt^2 - 16)\xt^2 + (4\rt^3 - 2\rt^2 + 6)\xt - 4\rt^3 + 8\rt^2 - 4\rt \\
\fh(\xt) &= (4\rt^2 - 16\rt)\xt^3 + (-\rt^4 + 4\rt^3 -10\rt^2 + 40\rt)\xt^2 \\
           &\qquad\qquad \qquad \qquad \qquad \qquad  + (2\rt^4 - 8\rt^3 + 4\rt^2 - 20\rt + 2)\xt + 4\rt^2.
\end{align*}
\end{prop}
\begin{proof}
By Proposition~\ref{minsidelength}, $\mu(r)=z(\xt_m)$, where $z(\xt)$ is given by equation~\eqref{side_length_equation}, and $\xt_m$ is the unique point in the interval $(1-1/\sqrt{2},1)$ such that $dz/d\xt=0$, or equivalently $dA/d\xt=0$ where $A(\xt)=(z(\xt))^2$.

It remains to show that $\xt_m$ is a root of \eqref{poly_equation}.  Let
\[
\fa(\xt) = (-2\xt + 4\rt)\sqrt{\fb(\xt)} + (4\rt - 2)\xt + \rt^4 - 4\rt^2,
\]
so that expansion in \eqref{side_length_equation} gives
\begin{multline*}
A(\xt) = 2\Big((\xt - \rt^2)\sqrt{\fa(\xt)} + (-\xt + 2\rt)\sqrt{\fb(\xt)} + \xt^2 \\
 + (-\rt^2 + 2\rt - 1)\xt + \rt^4 -2\rt^2\Big).
\end{multline*}
Note that
\begin{equation*}
\frac{d\fa}{d\xt} = \frac{(4\rt - 2)\sqrt{\fb(\xt)} + 4\xt^2 - (4\rt + 6)\xt + 4\rt}{\sqrt{\fb(\xt)}},
\end{equation*}
which leads to
\begin{multline*}
\frac{1}{2}\frac{dA}{d\xt} = \frac{(\xt- \rt^2)\left((2\rt-1)\sqrt{\fb(\xt)}+2\xt^2-(2\rt+3)\xt+2\rt\right)}{\sqrt{\fa(\xt)}\sqrt{\fb(\xt)}} + \sqrt{\fa(\xt)} \\
                                    + \frac{\xt^2-(2\rt+1)\xt+2\rt}{\sqrt{\fb(\xt)}} - \sqrt{\fb(\xt)} + 2\xt - \rt^2 + 2\rt - 1.
\end{multline*}
In view of the fact that $dA/d\xt=0$ at $x=\xt_m$,
set $\frac{1}{2} dA/d\xt = 0$, then multiply by $\sqrt{\fa(\xt)}\sqrt{\fb(\xt)}$ and separate terms with $\sqrt{\fa(\xt)}$ from the rest to obtain
\begin{multline*}
\left((2\xt - \rt^2 + 2\rt - 1)\sqrt{\fb(\xt)} + 2\xt^2 - (2\rt+3)\xt + 2\rt\right)\sqrt{\fa(\xt)} \\
= -\fc(\xt)\sqrt{\fb(\xt)} - \fd(\xt).
\end{multline*}
Note that $\fa(\xt) = \fe(\xt)\sqrt{\fb(\xt)} + \ff(\xt)$, so that squaring both sides and omitting ``$(\xt)$'' for readability results in
\begin{equation*}
\left(\fg\sqrt{\fb} + \fh\right)\left(\fe\sqrt{\fb} + \ff\right) = \fc^2\fb + 2\fc\fd\sqrt{\fb} + \fd^2.
\end{equation*}
Separating terms with $\sqrt{\fb}$ from the rest gives
\begin{equation*}
(\fe\fh + \ff\fg - 2\fc\fd)\sqrt{\fb} = \fc^2\fb + \fd^2 - \fe\fb\fg - \ff\fh.
\end{equation*}
Squaring both sides and rearranging leads to the final polynomial equation,
\begin{equation*}
\fb(\fe\fh + \ff\fg - 2\fc\fd)^2 - (\fc^2\fb + \fd^2 - \fe\fb\fg - \ff\fh)^2 = 0.
\end{equation*}
This polynomial has degree~10.  Terms of higher degree could arise only from $\fd^2$ and $\fe\fb\fg$, each of which has degree~6, but both $\fd^2$ and $\fe\fb\fg$ have $16\xt^6$ as their 6th degree term, so $\fd^2-\fe\fb\fg$ has degree~5.

In summary, $\xt_m$ is a root of the polynomial in \eqref{poly_equation}.
\end{proof}

\section{Nonexistence of a solution by radicals.}\label{algebraic_section}

In this section we prove that $\mu$ is not a radical function as defined in Section \ref{intro_section}. For convenience we work instead with the function $\lambda$ defined by $\lambda(c) = (\mu(c^2))^2$. It is intuitively clear that if $\mu$ is a radical function then $\lambda$ is also radical; a formal proof of this fact is given in a more general context in Lemma \ref{radical_lemma} below.  The proof of our main result, Theorem \ref{main_thm}, shows that $\lambda$ is not radical on any infinite subset of $[1,\infty)$. As a consequence we obtain the fact that $\mu$ cannot be radical on any such set. We refer the reader to \cite[Chapters 13, 14]{dummit-foote} for the algebraic background assumed in this section. 

\begin{lem}\label{radical_lemma}
Let $J$ be a nonempty subset of $[1,\infty)$, let $n\in\mathbb \Z^+$, and let $I=\sqrt[n]{J}$ be the set of positive $n$th roots of elements of $J$. Suppose that a function $f:J\to\R$ is radical, and define $g:I\to\R$ and $h:J\to\R$ by $g(c)=f(c^n)$ and $h(c)=(f(c))^n$. Then $g$ and $h$ are radical.\end{lem}
\begin{proof}
We begin by showing that $g$ is radical. Let $q\in\C[k,x]$ be a nonzero polynomial satisfying $q(c,f(c))=0$ for every $c\in J$, and such that there is a radical extension $R/\C(k)$ containing a splitting field $S$ of $q$. Let $\varphi:\C(k)\to\C(k)$ be the embedding induced by the map $k\mapsto k^n$, and let $\Omega$ be an algebraic closure of $\C(k)$. By basic field theory (see \cite[Chapter V, \S 2, Theorem 2.8]{lang}), we may extend the map $\varphi$ to an embedding $\varphi:R\to\Omega$. Defining $Q(k,x)=q(k^n,x)$ we have $Q(c,g(c))=0$ for every $c\in I$. Moreover, since $Q$ is the polynomial obtained by applying $\varphi$ to the coefficients of $q$, the above observations imply that $Q$ splits in the field $\varphi(S)$, which is contained in the radical extension $\varphi(R)/\C(k)$. Thus $g$ is a radical function.

Next we show that $h$ is radical. The argument will be given assuming that $q$ is monic and has no repeated root; the general case can be proved similarly. Let $\alpha_1,\ldots,\alpha_d$ be the roots of $q$ in $S$, and let 
\[Q(k,x)=(x-\alpha_1^n)\cdots(x-\alpha_d^n).\]
Note that since $q$ has coefficients in $\C[k]$, the same holds for $Q$. (The elementary symmetric functions of $\alpha_1^n,\ldots,\alpha_d^n$ are polynomials in the elementary symmetric functions of $\alpha_1,\ldots,\alpha_d$.) Moreover, as explained below, one can show that $Q(c,h(c))=0$ for every $c\in J$. Since $Q$ splits in the field $\C(k,\alpha_1^n,\ldots,\alpha_d^n)\subseteq S\subseteq R$, this implies that $h$ is radical.

Fixing $c\in J$, the fact that $Q(c,h(c))=0$ can be seen heuristically first: Since
\[q(k,x)=(x-\alpha_1)\cdots(x-\alpha_d)\]
and $q(c,f(c))=0$, we must have $f(c)=\alpha_i(c)$ for some $i$, and thus $Q(c,h(c))=0$. This argument can be made rigorous by extending the map $k\mapsto c$ to a ring homomorphism $\C[k,\alpha_1,\ldots,\alpha_d]\to\C$, so that $\alpha_i(c)$ is well-defined. We refer the interested reader to Section 3 in \cite[Chapter VII]{lang} for the necessary tools.
\end{proof}

Next we prove two preliminary results needed to show that $\lambda$ is not radical.

\begin{lem}\label{p_irred}
Let $p(k,x)$ be the polynomial defined by~\eqref{poly_equation} considering $k$ and $x$ as indeterminates. As an element of the ring $\C(k)[x]$, the polynomial $p$ is irreducible and has Galois group isomorphic to the symmetric group $S_{10}$.
\end{lem}
\begin{proof}
We rely on a computation carried out using the computer algebra system \textsc{Magma}; the code for our computation is available in the supplemental online material. Constructing $p(k,x)$ as an element of the ring $\Q(k)[x]$, we use Sutherland's algorithm \cite{sutherland} to compute a permutation representation of the Galois group of $p$ over $\C(k)$, and we obtain the group $S_{10}$. It follows that the Galois group acts transitively on the roots of $p$, so $p$ is irreducible over $\C(k)$.
\end{proof}

By Proposition \ref{minsidelength} we may regard $x_m$ as a function of $r$ defined on the interval $[1,\infty)$, and moreover, we have
\begin{equation}\label{mu_z_xm}
\mu(r)=z(x_m(r))\quad\text{for all}\quad r\ge 1.
\end{equation}

For convenience we will make the change of variable $k=\sqrt r$ and work instead with the function $\xi:[1,\infty)\to\R$ be defined by $\xi(k)=x_m(k^2)$. From Proposition \ref{polyequation} we deduce that
\begin{equation}\label{P_xi_ppty}
p(k,\xi(k))=0\quad\text{for all}\quad k\in [1,\infty).
\end{equation}
Furthermore, \eqref{mu_z_xm} implies that $\lambda(k)=\left(z(\xi(k))\right)^2$. Hence, writing $\xi$ for $\xi(k)$, we have
\[\lambda(k)=\xi^2 + \left(k^2 - \xi - \sqrt{k^4 - \left(2k - \xi - \sqrt{2\xi - \xi^2}\right)^2}\right)^2.\]

By manipulating the equation above we obtain a polynomial $h\in\Q[k,x,y]$ with the property that
\begin{equation}\label{h_xi_lambda_property}
h(k,\xi(k),\lambda(k))=0\quad\text{for all}\quad k\in [1,\infty).
\end{equation}
Explicitly, $h$ is given by the formula
\begin{multline*}h(k,x,y)=\left((y - x^2 - c_1^2 - c_2 + c_3^2 + c_4)^2 + 4c_3^2c_4 - 4c_1^2c_2 + 4c_1^2c_3^2 + 4c_1^2c_4\right)^2 \\
- c_4\left(8c_1^2c_3 + 4c_3(y - x^2 - c_1^2 - c_2 + c_3^2 + c_4)\right)^2,
\end{multline*}
where $c_1=k^2-x$, $c_2=k^4$, $c_3=2k-x$, and $c_4=2x-x^2$.

\begin{lem}\label{alpha_beta_extensions}
Let $\Omega$ be an algebraic closure of the field $\C(k)$. Suppose that $\alpha,\beta\in\Omega$ satisfy $p(k,\alpha)=h(k,\alpha,\beta)=0$. Then $\C(k,\alpha)\subseteq\C(k,\beta)$.
\end{lem}
\begin{proof}
We rely on a number of computations in \textsc{Magma}; the code used for all computations is available in the supplemental online material. Let $F=\Q(k,\alpha)$. To prove the lemma it suffices to show that $F\subseteq\Q(k,\beta)$. Regarding $p$ as an element of the ring $\Q(k)[x]$, note that $p$ is irreducible by Lemma \ref{p_irred}, and $\alpha$ is a root of $p$ by hypothesis, so we may identify $F$ with the field $\Q(k)[x]/(p(k,x))$. Constructing $F$ in \textsc{Magma} and factoring\footnote{The algorithm used by \textsc{Magma} to factor polynomials over algebraic function fields is discussed in \cite{trager}.} the polynomial $h(k,\alpha, y)$ over~$F$, we find that this polynomial has two roots in $F$ and two roots that are quadratic over $F$. Note that $\beta$ must be one of these four roots since $h(k,\alpha,\beta)=0$.

Suppose that $\beta\in F$. Computing the minimal polynomial of $\beta$ over $\Q(k)$ we obtain a polynomial of degree 10; thus $[\Q(k,\beta):\Q(k)]=10$. Since $\beta\in F$ and $$[F:\Q(k)]=\deg(p)=10,$$ this implies that $F=\Q(k,\beta)$. In particular, $F\subseteq\Q(k,\beta)$ as desired.

Now suppose that $\beta$ is quadratic over $F$. Then a minimal polynomial computation shows that $[\Q(k,\beta):\Q(k)]=20$. Since $[F(\beta):F]=2$ and $[F:\Q(k)]=10$, this implies that $F(\beta)=\Q(k,\beta)$, so again $F\subseteq\Q(k,\beta)$.
\end{proof}

We can now prove the main theorem of this article.

\begin{thm}\label{main_thm}
There is no infinite subset $J\subseteq [1,\infty)$ such that $\mu:J\to\R$ is radical.
\end{thm}
\begin{proof}
As above, let $\Omega$ denote an algebraic closure of the field $\C(k)$. By Lemma \ref{radical_lemma}, in order to prove the theorem it suffices to show that $\lambda$ is not radical on any infinite subset of $[1,\infty)$. Suppose for contradiction that $I\subseteq [1,\infty)$ is an infinite set such that $\lambda:I\to\R$ is radical. Then there is a nonzero polynomial $q(k,y)\in\C[k,y]$ whose Galois group over $\C(k)$ is solvable, and such that
\begin{equation}\label{q_mu_property}
q(k, \lambda(k))=0\quad\text{for all}\quad k \in I.
\end{equation}
Regarding $q$ and $h$ as elements of the ring $\C[k,x,y]$, let
\[f(k,x)=\Res_y\left(h(k,x,y), q(k,x,y)\right),\]
where $\Res_y$ denotes the resultant as polynomials in~$y$. By \eqref{h_xi_lambda_property} and \eqref{q_mu_property}, for every $k\in I$ the polynomials $h(k,\xi(k),y)$ and $q(k,\xi(k),y)$ have a common root, namely $\lambda(k)$; hence
\begin{equation}\label{f_xi_property}
f(k,\xi(k))=0\quad\text{for all}\quad k\in I.
\end{equation}

Similarly, letting
\[g(k)=\Res_x\left(f(k,x),p(k,x)\right)\in\C[k],\]
equations \eqref{P_xi_ppty} and \eqref{f_xi_property} imply that $g(k)=0$ for every $k\in I$. Since $I$ is an infinite set, we must have $g=0$. Therefore, $f$ and $p$ have a common root $\alpha\in\Omega$. Given that $f(k,\alpha)=0$, the definition of $f$ implies that $q(k,y)$ and $h(k,\alpha,y)$ have a common root $\beta\in\Omega$. Note that the assumptions in Lemma \ref{alpha_beta_extensions} are satisfied.

Let $N\subset\Omega$ be the splitting field of $q(k,y)$ over $\C(k)$ and let $F=\C(k,\alpha)$. By Lemma \ref{alpha_beta_extensions} we have $F\subseteq\C(k,\beta)$ and therefore $F\subseteq N$. Letting $L\subset\Omega$ be the splitting field of $p(k,x)$ over $\C(k)$, we have $L\subseteq N$ since $F\subseteq N$ and $L$ is the Galois closure of the extension $F/\C(k)$. Since the extension $L/\C(k)$ is Galois, the group $\Gal(L/\C(k))$ is a quotient of $\Gal(N/\C(k))$. The definition of $q(k,y)$ implies that the latter group is solvable, so the former is, too. This contradicts Lemma \ref{p_irred} (since the group $S_{10}$ is not solvable), and thus completes the proof of the theorem.
\end{proof}

\begin{acknowledgment}{Acknowledgments.}
JEH received support from the Colby College Research Grant Program. The authors thank Gerardo Lafferriere of the Fariborz Maseeh Department of Mathematics and Statistics at Portland State University for welcoming JEH as a Visiting Scholar during the final stages of this project.
\end{acknowledgment}

\begin{biog}
\item[Jan E. Holly]
\begin{affil}
Department of Mathematics and Statistics, Colby College, Waterville, ME 04901\\
Jan.Holly@colby.edu
\end{affil}

\item[David Krumm]
\begin{affil}
Mathematics Department, Reed College, Portland, OR 97202\\
dkrumm@reed.edu
\end{affil}
\end{biog}
\vfill\eject

\end{document}